\documentclass[11pt]{amsart}
\usepackage{amsmath,amssymb,mathrsfs}
\newtheorem{theorem}{Theorem}[section]
\newtheorem{proposition}[theorem]{Proposition}
\newtheorem{corollary}[theorem]{Corollary}
\newtheorem{lemma}[theorem]{Lemma}
\theoremstyle{definition}
\newtheorem{definition}[theorem]{Definition}
\newtheorem{remark}[theorem]{Remark}

\begin{document} 

\title[Positive mass theorem in arbitrary dimension]{A dimension descent scheme for the positive mass theorem in arbitrary dimension}
\author{Simon Brendle and Yipeng Wang}
\address{Columbia University \\ 2990 Broadway \\ New York NY 10027 \\ USA}
\address{Columbia University \\ 2990 Broadway \\ New York NY 10027 \\ USA}
\thanks{The first author was supported by the National Science Foundation under grant DMS-2403981 and by the Simons Foundation.}
\maketitle
\begin{abstract}
We describe how the Schoen-Yau proof of the positive mass theorem can be extended to arbitrary dimensions. To overcome the problem of singularities, we propose a new inductive scheme. To carry out the inductive step, we use a combination of several techniques, including the shielding principle of Lesourd-Unger-Yau, as well as a conformal blow-up argument in the spirit of Bi-Hao-He-Shi-Zhu. Our arguments also rely on the Cheeger-Naber bound for the Minkowski dimension of the singular set.
\end{abstract}

\tableofcontents

\section{Introduction}

We begin with several definitions.

\begin{definition}
\label{AF.end}
Let $n \geq 3$ be an integer. Let $(M,g)$ be a complete Riemannian manifold of dimension $n$. We say that $(M,g)$ has an asymptotically flat end if there exists a compact domain $K \subset M$ with smooth boundary and a connected component $E_0$ of $M \setminus K$ such that $E_0$ is diffeomorphic to the complement of the unit ball in $\mathbb{R}^n$. Moreover, we require that there exist real numbers $\alpha$ and $\delta>0$ such that 
\[|\bar{D}^m(g - (1+\alpha \, r^{2-n}) \, \bar{g})|_{\bar{g}} \leq C(m) \, r^{2-n-m-2\delta}\] 
at each point in $E_0$ and for every nonnegative integer $m$. Here, $\bar{g}$ denotes the Euclidean metric on the asymptotically flat end $E_0$, $\bar{D}^m$ denotes the covariant derivative of order $m$ with respect to $\bar{g}$, and $r = \sqrt{x_1^2+\hdots+x_n^2}$ denotes the radial coordinate on the asymptotically flat end $E_0$.
\end{definition}

\begin{definition}
\label{mass.of.AF.end}
Let $n \geq 3$ be an integer. Let $(M,g)$ be a complete Riemannian manifold of dimension $n$ with an asymptotically flat end. We define the mass of $(M,g)$ to be $(n-1)\alpha$, where $\alpha$ denotes the coefficient in the asymptotic expansion of the metric.
\end{definition}

\begin{definition}
\label{dataset}
Let $n \geq 3$ be an integer. An $n$-dataset consists of a complete Riemannian manifold $(M,g)$ of dimension $n$ together with positive smooth functions $\rho$ and $Q$ satisfying the following conditions: 
\begin{itemize}
\item The manifold $(M,g)$ has an asymptotically flat end $E_0$. 
\item If $n=3$, we assume in addition that $M \setminus E_0$ is a bounded subset of $(M,g)$. In other words, if $n=3$, we assume that $(M,g)$ has no ends other than $E_0$.
\item There exist real numbers $\beta$ and $\delta>0$ such that 
\[|\bar{D}^m(\rho - (1+\beta \, r^{2-n}))|_{\bar{g}} \leq C(m) \, r^{2-n-m-2\delta}\] 
and 
\[|\bar{D}^m Q|_{\bar{g}} \leq C(m) \, r^{-n-m-2\delta}\] 
at each point in $E_0$ and for every nonnegative integer $m$. 
\item We have 
\begin{align*}
&\int_M \rho \, |df|^2 + \frac{1}{2} \int_M \rho \, \Big ( R - 2 \, \Delta \log \rho - \frac{n+1}{n+2} \, |d \log \rho|^2 \Big ) \, f^2 \\ 
&\geq \int_M \rho \, Q \, f^2 
\end{align*}
for every smooth test function $f$ with the property that the set $\{f \neq 0\} \setminus E_0$ is bounded and there exists a constant $a$ such that the set $\{f \neq a\} \cap E_0$ is bounded. In view of our decay conditions, the functions $R$, $\Delta \log \rho$, $|d\log \rho|^2$, and $Q$ belong to $L^1(E_0)$, so the integrals are well-defined.
\end{itemize}
\end{definition}

\begin{definition}
\label{mass.of.dataset}
Let $n \geq 3$ be an integer, and let $(M,g,\rho,Q)$ be an $n$-dataset. We define the mass of the $n$-dataset $(M,g,\rho,Q)$ to be $(n-1)\alpha + 2\beta$, where $\alpha$ denotes the coefficient in the asymptotic expansion of the metric and $\beta$ denotes the coefficient in the asymptotic expansion of $\rho$.
\end{definition}

We now state the main result of this paper.

\begin{theorem}
\label{pmt}
Let $n \geq 3$ be an integer, and let $(M,g,\rho,Q)$ be an $n$-dataset. Then the mass (in the sense of Definition \ref{mass.of.dataset}) of the $n$-dataset $(M,g,\rho,Q)$ is strictly positive.
\end{theorem}

As a special case, we obtain the following result.

\begin{corollary}
\label{special.case}
Let $n \geq 3$ be an integer. Suppose that $(M,g)$ is a complete Riemannian manifold with an asymptotically flat end. If the scalar curvature of $g$ is positive at each point in $M$, then the mass (in the sense of Definition \ref{mass.of.AF.end}) of $(M,g)$ is strictly positive.
\end{corollary}

Corollary \ref{special.case} follows from Theorem \ref{pmt} by putting $\rho=1$ and $Q=\frac{1}{2} \, R$. 

In their groundbreaking works \cite{Schoen},\cite{Schoen-Yau1}, Schoen and Yau proved the positive mass theorem for asymptotically flat manifolds of dimension $n \leq 7$. Lesourd, Unger, and Yau \cite{Lesourd-Unger-Yau} extended the positive mass theorem to manifolds which have one asymptotically flat end and in addition have other arbitrary ends. As part of their work, Lesourd, Unger, and Yau introduced a shielding principle, which plays a central role in our work. Chodosh, Mantoulidis, Schulze, and Wang \cite{Chodosh-Mantoulidis-Schulze-Wang} have verified the positive mass theorem up to dimension $11$, and Bi, Hao, He, Shi, and Zhu \cite{Bi-Hao-He-Shi-Zhu} recently gave a proof of the positive mass theorem up to dimension $19$. Finally, Schoen and Yau \cite{Schoen-Yau2} and Lohkamp \cite{Lohkamp} have proposed proofs of the positive mass theorem in arbitrary dimension.

The proof of Theorem \ref{pmt} proceeds by induction on $n$. For $n=3$, Theorem \ref{pmt} can be reduced to the classical positive mass theorem of Schoen and Yau. 

We now give an overview of the proof of the inductive step. Suppose that $n \geq 4$ and $(M,g,\rho,Q)$ is an $n$-dataset with nonpositive mass. Following Lesourd-Unger-Yau, we construct an open domain $E$ together with smooth functions $\Phi$ and $\hat{Q}$ satisfying the following conditions:
\begin{itemize}
\item The closure of $E_0$ is contained in $E$.
\item The complement $E \setminus E_0$ is a bounded subset of $(M,g)$. 
\item $\Phi=0$ and $\hat{Q} = \frac{1}{2} \, Q$ at each point in $E_0$.
\item $\Phi \leq 0$ and $\hat{Q} > 0$ at each point in $E$.
\item $\Phi \to -\infty$ on the boundary $\partial E$. 
\item $Q + \frac{1}{2} \, \Phi^2 - 2 \, |d\Phi| \geq 2\hat{Q}$ at each point in $E$. 
\end{itemize}
Note that $E$ has compact, smooth boundary and one asymptotically flat end.

In the next step, we slightly enlarge the domain $E$. On the enlarged domain $\hat{E}$, we construct a positive solution $\hat{v}$ of the linear PDE 
\[-\Delta \hat{v} - \langle d\log \rho,d\hat{v} \rangle + \frac{1}{2} \, \Big ( R - 2 \, \Delta \log \rho - \frac{n+1}{n+2} \, |d\log \rho|^2 - Q \Big ) \, \hat{v} = 0\] 
with Dirichlet boundary condition on $\partial \hat{E}$. We then restrict $\hat{v}$ to the smaller domain $E$, and define $\hat{\rho} = \rho \, \hat{v}$. 

In the next step, we construct a $\mu$-bubble in $E$. This $\mu$-bubble may have singularities. We denote by $\Sigma$ the regular part of the $\mu$-bubble. Then $\Sigma$ is a smooth hypersurface in $E$ satisfying 
\[H_\Sigma + \langle \nabla \log \hat{\rho},\nu_\Sigma \rangle = \Phi.\] 
To construct this $\mu$-bubble we need barriers near infinity, as well as barriers near the boundary $\partial E$. To construct the barriers near infinity, we use the fact that the $n$-dataset $(M,g,\rho,Q)$ has nonpositive mass. To construct the barriers near $\partial E$, we use the fact that $\Phi \to -\infty$ on $\partial E$. 

The hypersurface $\Sigma$ with its induced metric $\hat{g}$ can be viewed as an incomplete manifold of dimension $n-1$ with an asymptotically flat end. Moreover, $\Sigma$ satisfies a stability inequality. By combining the stability inequality for $\Sigma$ with a generalization of the famous Schoen-Yau identity, we conclude that a certain quadratic form on $\Sigma$ is positive. 

In the last step, we construct a conformal metric $\tilde{g} = w^{\frac{n+1}{n-3}} \, \hat{g}$ on $\Sigma$. The conformal factor $w$ is obtained by restricting a suitable function on ambient space to $\Sigma$. This function blows up at a controlled rate near the singular set, thereby ensuring that the metric $\tilde{g}$ is complete. Importantly, the conformal factor can be chosen in such a way that the positivity of the quadratic form is preserved. This allows us to construct an $(n-1)$-dataset with mass equal to $0$, thereby completing the inductive step.

\section{Proof of Theorem \ref{pmt} for $n=3$}

Throughout this section, we assume that $(M,g,\rho,Q)$ is a $3$-dataset. Let $E_0$ denote the asymptotically flat end of $(M,g)$. By assumption, $M \setminus E_0$ is a bounded subset of $(M,g)$. 

\begin{proposition}
\label{quadratic.form.3D}
Suppose that $a$ is a constant and $F$ is a smooth function on $M$ such that $F = a \, \rho^{\frac{1}{2}}$ near infinity. Then 
\[\int_M |dF|^2 + \frac{1}{8} \int_M R \, F^2 - \frac{1}{4} \int_M Q \, F^2 \geq -\pi \beta a^2,\] 
where $\beta$ denotes the coefficient in the asymptotic expansion of $\rho$.
\end{proposition}

\textbf{Proof.}
Let $f = \rho^{-\frac{1}{2}} \, F$. Then $f = a$ near infinity. Since $(M,g,\rho,Q)$ is a $3$-dataset, we know that 
\begin{equation}
\label{3D.inequality}
\int_M \rho \, |df|^2 + \frac{1}{2} \int_M \rho \, \Big ( R - 2 \, \Delta \log \rho - \frac{4}{5} \, |d \log \rho|^2 \Big ) \, f^2 \geq \int_M \rho \, Q \, f^2. 
\end{equation}
On the other hand, it follows from the divergence theorem that 
\begin{equation} 
\label{divergence.theorem}
\int_M \text{\rm div}(f^2 \, d\rho) = -4\pi \beta a^2. 
\end{equation}
Adding (\ref{3D.inequality}) and (\ref{divergence.theorem}) gives 
\begin{align*}
&\int_M \rho \, |df|^2 + 2 \int_M f \, \langle d\rho,df \rangle + \frac{3}{5} \int_M \rho^{-1} \, |d\rho|^2 \, f^2 + \frac{1}{2} \int_M \rho \, R \, f^2 \\ 
&\geq -4\pi \beta a^2 + \int_M \rho \, Q \, f^2. 
\end{align*}
Using the pointwise inequality 
\begin{align*} 
&\rho \, |df|^2 + 2f \, \langle d\rho,df \rangle + \frac{3}{5} \, \rho^{-1} \, |d\rho|^2 \, f^2 \\ 
&= \frac{8}{3} \, \rho \, \Big | df + \frac{1}{2} \, \rho^{-1} \, f \, d\rho \Big |^2 - \frac{5}{3} \, \rho \, \Big | df + \frac{1}{5} \, \rho^{-1} \, f \, d\rho \Big |^2 \\ 
&\leq 4 \, \rho \, \Big | df + \frac{1}{2} \, \rho^{-1} \, f \, d\rho \Big |^2,
\end{align*}
we obtain 
\[4 \int_M \rho \, \Big | df + \frac{1}{2} \, \rho^{-1} \, f \, d\rho \Big |^2 + \frac{1}{2} \int_M \rho \, R \, f^2 \geq -4\pi \beta a^2 + \int_M \rho \, Q \, f^2.\] 
Since $F = \rho^{\frac{1}{2}} \, f$, we conclude that 
\[4 \int_M |dF|^2 + \frac{1}{2} \int_M R \, F^2 \geq -4\pi\beta a^2 + \int_M Q \, F^2.\] 
This completes the proof of Proposition \ref{quadratic.form.3D}. \\

In the next step, we construct a solution of a certain linear PDE. To that end, we follow the arguments in Eichmair-Huang-Lee-Schoen \cite{Eichmair-Huang-Lee-Schoen} and Carlotto \cite{Carlotto}. Let us fix a nonnegative smooth function $\omega$ such that $\omega$ is supported in $E_0$ and $\omega = r^{-2}$ near infinity. By Hardy's inequality, we can find a positive constant $\kappa$ such that 
\begin{equation} 
\label{Hardy.inequality.3D}
\int_M |dF|^2 + \frac{1}{8} \int_M Q \, F^2 \geq \kappa \int_M \omega \, F^2 
\end{equation} 
for every smooth function $F$ on $M$ that vanishes near infinity. Since the function $4\kappa \, \omega + R$ is positive near infinity, we can find a large constant $\Lambda>2$ such that 
\begin{equation} 
\label{choice.of.Lambda.3D}
(2\Lambda-4) \, Q + 4\kappa \, \omega + R \geq 0 
\end{equation}
at each point on $M$.

\begin{proposition}[Coercivity]
\label{coercivity.3D}
Suppose that $F$ is a smooth function on $M$ that vanishes near infinity. Then 
\[\int_M |dF|^2 + \frac{1}{8} \int_M R \, F^2 \geq \frac{1}{8\Lambda} \int_M (Q + 4\kappa \, \omega) \, F^2.\] 
\end{proposition}

\textbf{Proof.} 
We compute 
\begin{align*}
&\int_M |dF|^2 + \frac{1}{8} \int_M R \, F^2 - \frac{1}{8\Lambda} \int_M (Q + 4\kappa \, \omega) \, F^2 \\ 
&= \frac{\Lambda-1}{\Lambda} \, \bigg ( \int_M |dF|^2 + \frac{1}{8} \int_M R \, F^2 - \frac{1}{4} \int_M Q \, F^2 \bigg ) \\ 
&+ \frac{1}{\Lambda} \, \bigg ( \int_M |dF|^2 + \frac{1}{8} \int_M Q \, F^2 - \kappa \int_M \omega \, F^2 \bigg ) \\ 
&+ \frac{1}{8\Lambda} \int_M \Big ( (2\Lambda-4) \, Q + 4\kappa \, \omega + R \Big ) \, F^2. 
\end{align*} 
The first term on the right hand side is nonnegative by Proposition \ref{quadratic.form.3D}. The second term on the right hand side is nonnegative by the Hardy inequality (\ref{Hardy.inequality.3D}). The third term on the right hand side is nonnegative in view of (\ref{choice.of.Lambda.3D}). This completes the proof of Proposition \ref{coercivity.3D}. \\

Let $\mathcal{H}$ denote the set of all functions $F \in H_{\text{\rm loc}}^1(M)$ with the property that 
\[\int_M |dF|^2 + \int_M (Q+\omega) \, F^2 < \infty.\] 
We define 
\[\|F\|_{\mathcal{H}}^2 = \int_M |dF|^2 + \int_M (Q+\omega) \, F^2\] 
for all $F \in \mathcal{H}$. Since $|R| \leq O(r^{-3-2\delta})$, we know that $R \in L^1(M)$. Moreover, $R \, F \in L^1(M)$ and $R \, F^2 \in L^1(M)$ for all $F \in \mathcal{H}$. 

\begin{proposition}
\label{minimizer.of.quadratic.form.3D}
We can find a function $v \in \mathcal{H}$ with the property that $v$ minimizes the functional 
\[\int_M |dv|^2 + \frac{1}{8} \int_M R \, v^2 + \frac{1}{4} \int_M R \, v\] 
among all functions $v \in \mathcal{H}$. Moreover, we can choose $v$ so that $1+v \geq 0$. 
\end{proposition}

\textbf{Proof.} 
It follows from Proposition \ref{coercivity.3D} and a standard approximation argument that 
\[\int_M |dF|^2 + \frac{1}{8} \int_M R \, F^2 \geq \frac{1}{8\Lambda} \int_M (Q + 4\kappa \, \omega) \, F^2\] 
for all $F \in \mathcal{H}$. Using this coercivity property, the existence of a minimizer follows easily. By replacing $v$ by $|1+v| - 1$, we can arrange that $1+v$ is nonnegative. This completes the proof of Proposition \ref{minimizer.of.quadratic.form.3D}. \\

Let $v$ denote the minimizer constructed in Proposition \ref{minimizer.of.quadratic.form.3D}. By elliptic regularity theory, $v$ is a smooth solution of the PDE
\begin{equation} 
\label{pde.for.v.3D}
-\Delta v + \frac{1}{8} \, R \, (1+v) = 0. 
\end{equation}
Since $v \in \mathcal{H}$, the function $1+v$ does not vanish identically. Moreover, the function $1+v$ is nonnegative. Using the strict maximum principle, we conclude that the function $1+v$ is strictly positive everywhere. 

Our assumptions imply that 
\[|\bar{D}^m R|_{\bar{g}} \leq O(r^{-3-m-2\delta})\] 
for every nonnegative integer $m$. Standard results for linear PDE \cite{Meyers} imply that there exists a real number $\gamma$ such that 
\[|\bar{D}^m(v - \gamma \, r^{-1})|_{\bar{g}} \leq O(r^{-1-m-2\hat{\delta}})\] 
for every nonnegative integer $m$, where $\hat{\delta} \in (0,\delta)$. 

\begin{proposition}
\label{leading.term.in.asymptotic.expansion.of.v.3D}
We have $4\gamma < \beta$, where $\beta$ denotes the coefficient in the asymptotic expansion of $\rho$ and $\gamma$ denotes the coefficient in the asymptotic expansion of $v$. 
\end{proposition}

\textbf{Proof.} 
Proposition \ref{quadratic.form.3D} implies that 
\[\int_M |dF|^2 + \frac{1}{8} \int_M R \, F^2 - \frac{1}{4} \int_M Q \, F^2 \geq -\pi \beta\] 
for every smooth function $F$ on $M$ with the property that $F=\rho^{\frac{1}{2}}$ near infinity. Note that $\rho = 1+O(r^{-1})$ and $1+v = 1+O(r^{-1})$, and we have corresponding estimates for all the higher derivatives. By a standard approximation argument, the preceding inequality holds for the function $F=1+v$. This gives 
\[\int_M |dv|^2 + \frac{1}{8} \int_M R \, (1+v)^2 - \frac{1}{4} \int_M Q \, (1+v)^2 \geq -\pi \beta.\] 
On the other hand, using (\ref{pde.for.v.3D}), we obtain 
\[\int_M |dv|^2 + \frac{1}{8} \int_M R \, (1+v)^2 = \int_M \text{\rm div}((1+v) \, dv) = -4\pi \gamma.\] 
Finally, since the function $Q$ is strictly positive and the function $1+v$ is strictly positive, we know that 
\[\int_M Q \, (1+v)^2 > 0.\]
Putting these facts together, we conclude that $-4\pi \gamma > -\pi \beta$. This completes the proof of Proposition \ref{leading.term.in.asymptotic.expansion.of.v.3D}. \\ 

It follows from (\ref{pde.for.v.3D}) that the conformal metric $(1+v)^4 \, g$ has zero scalar curvature. Moreover, 
\[(1+v)^4 \, g = (1+(\alpha+4\gamma) \, r^{-1}) \, \bar{g} + O(r^{-1-\hat{\delta}})\] 
near infinity. We now apply the classical positive mass theorem of Schoen and Yau to the conformal metric $(1+v)^4 \, g$. This implies $\alpha + 4\gamma \geq 0$. Combining this inequality with Proposition \ref{leading.term.in.asymptotic.expansion.of.v.3D}, we conclude that $\alpha + \beta > 0$. Therefore, the $3$-dataset $(M,g,\rho,Q)$ has strictly positive mass in the sense of Definition \ref{mass.of.dataset}. This completes the proof of Theorem \ref{pmt} in the special case $n=3$.

\section{Proof of Theorem \ref{pmt} for $n \geq 4$}

In this section, we complete the proof of Theorem \ref{pmt}. Let us fix an integer $n \geq 4$. We assume that Theorem \ref{pmt} holds for all $(n-1)$-datasets. We will show that Theorem \ref{pmt} holds for all $n$-datasets. We argue by contradiction. Suppose that $(M,g,\rho,Q)$ is an $n$-dataset with nonpositive mass. In other words,  
\begin{equation} 
\label{assumption}
(n-1)\alpha + 2\beta \leq 0, 
\end{equation} 
where $\alpha$ denotes the coefficient in the asymptotic expansion of the metric and $\beta$ denotes the coefficient in the asymptotic expansion of $\rho$. Let $E_0$ denote the asymptotically flat end of $M$. 

\subsection{A result from Lesourd-Unger-Yau's work} In this subsection, we recall an important construction from the work of Lesourd-Unger-Yau. 

\begin{lemma}[cf. Lesourd-Unger-Yau \cite{Lesourd-Unger-Yau}, Proposition 3.1]
\label{Phi}
We can find an open, connected domain $E$ with smooth boundary, a smooth function $\Phi$ defined on $E$, and a smooth function $\hat{Q}$ defined on $E$ with the following properties: 
\begin{itemize}
\item The closure of $E_0$ is contained in $E$.
\item The complement $E \setminus E_0$ is a bounded subset of $(M,g)$. 
\item $\Phi=0$ and $\hat{Q} = \frac{1}{2} \, Q$ at each point in $E_0$.
\item $\Phi \leq 0$ and $\hat{Q} > 0$ at each point in $E$.
\item $\Phi \to -\infty$ on the boundary $\partial E$. 
\item $Q + \frac{1}{2} \, \Phi^2 - 2 \, |d\Phi| \geq 2\hat{Q}$ at each point in $E$. 
\end{itemize}
\end{lemma}

\textbf{Proof.} 
Let us fix positive real numbers $s_0$ and $s_1$ such that $s_1 \geq s_0$ and 
\[Q(x) > \frac{128}{s_1 s_0}\] 
for each point $x \in \mathcal{N}_{(M,g)}(E_0,2s_0) \setminus E_0$. Let us fix a smooth function $\chi: \mathbb{R} \to [0,1]$ such that $\chi=0$ on $[0,\frac{1}{2}]$, $\chi=1$ on $[1,\infty)$, and $0 \leq \chi' \leq 3$. We define a smooth function $\varphi: [0,s_1+s_0) \to (-\infty,0]$ by 
\[\varphi(s) = -\frac{8}{s_1+s_0-s} \, \chi \Big ( \frac{s}{s_0} \Big )\] 
for all $s \in [0,s_1+s_0)$. Clearly, $\varphi$ is monotone decreasing, and $\varphi(s) \to -\infty$ as $s \to s_1+s_0$. Moreover, 
\begin{align*} 
|\varphi'(s)| 
&= \frac{8}{(s_1+s_0-s)s_0} \, \chi' \Big ( \frac{s}{s_0} \Big ) + \frac{8}{(s_1+s_0-s)^2} \, \chi \Big ( \frac{s}{s_0} \Big ) \\ 
&\leq \frac{24}{s_1s_0} + \frac{8}{s_1^2} \\ 
&\leq \frac{32}{s_1s_0}
\end{align*}
for all $s \in [0,s_0)$ and 
\[|\varphi'(s)| = \frac{1}{8} \, \varphi(s)^2\] 
for all $s \in [s_0,s_1+s_0)$.

We can find a nonnegative smooth function $\sigma$ such that $\sigma = 0$ on $E_0$, $|\sigma - d_{(M,g)}(\cdot,E_0)| < s_0$, and $|d\sigma| \leq 2$. We may further assume that $s_1+s_0$ is a regular value of $\sigma$. We claim that 
\begin{equation} 
\label{key.inequality} 
Q(x) + \frac{1}{2} \, \varphi(\sigma(x))^2 - 4 \, |\varphi'(\sigma(x))| > 0 
\end{equation} 
for each point $x \in \{\sigma < s_1+s_0\}$. To prove this, we distinguish three cases: 

\textit{Case 1:} Suppose first that $x \in E_0$. In this case, 
\[Q(x) + \frac{1}{2} \, \varphi(\sigma(x))^2 - 4 \, |\varphi'(\sigma(x))| = Q(x) > 0.\] 

\textit{Case 2:} Suppose that $x \in \{\sigma < s_0\} \setminus E_0$. Then $x \in \mathcal{N}_{(M,g)}(E_0,2s_0) \setminus E_0$. Using the inequality $Q(x) > \frac{128}{s_1s_0}$, we obtain 
\[Q(x) + \frac{1}{2} \, \varphi(\sigma(x))^2 - 4 \, |\varphi'(\sigma(x))| > \frac{128}{s_1s_0} - 4 \, |\varphi'(\sigma(x))| \geq 0.\] 

\textit{Case 3:} Suppose that $x \in \{s_0 \leq \sigma < s_1+s_0\}$. Using the inequality $Q(x) > 0$, we obtain 
\[Q(x) + \frac{1}{2} \, \varphi(\sigma(x))^2 - 4 \, |\varphi'(\sigma(x))| > \frac{1}{2} \, \varphi(\sigma(x))^2 - 4 \, |\varphi'(\sigma(x))| = 0.\] 
This proves (\ref{key.inequality}). 

Let $E$ denote the connected component of the set $\{\sigma < s_1+s_0\}$ that contains the set $E_0$. Then $E \subset \mathcal{N}_{(M,g)}(E_0,s_1+2s_0)$. In particular, the set $E \setminus E_0$ is a bounded subset of $(M,g)$. We define a smooth function $\Phi$ on $E$ by 
\[\Phi(x) = \varphi(\sigma(x))\] 
for $x \in E$. Clearly, $\Phi = 0$ at each point in $E_0$, $\Phi \leq 0$ at each point in $E$, and $\Phi \to -\infty$ on the boundary $\partial E$. Finally, we define a smooth function $\hat{Q}$ on $E$ by 
\[\hat{Q}(x) = \frac{1}{2} \, Q(x) + \frac{1}{4} \, \varphi(\sigma(x))^2 + 2 \, \varphi'(\sigma(x))\] 
for $x \in E$. It follows from (\ref{key.inequality}) that $\hat{Q}(x) > 0$ for each point $x \in E$. Moreover, 
\begin{align*} 
&Q(x) + \frac{1}{2} \, \Phi(x)^2 - 2 \, |d\Phi(x)| - 2 \, \hat{Q}(x) \\ 
&= Q(x) + \frac{1}{2} \, \varphi(\sigma(x))^2 - 2 \, |\varphi'(\sigma(x))| \, |d\sigma(x)| - 2 \, \hat{Q}(x) \\ 
&= 2 \, |\varphi'(\sigma(x))| \, (2 - |d\sigma(x)|) \\ 
&\geq 0 
\end{align*} 
for each point $x \in E$. This completes the proof of Lemma \ref{Phi}. 

\begin{remark}
In Lemma \ref{Phi}, we allow the possibility that $\partial E = \emptyset$. 
\end{remark}

In the following, we fix an open, connected domain $\hat{E}$ with smooth boundary such that the closure of $E$ is contained in $\hat{E}$ and the complement $\hat{E} \setminus E$ is a bounded subset of $(M,g)$. \textbf{From now on, we will work exclusively on the closure of the domain $\hat{E}$.}

\subsection{Solving a linear PDE with Dirichlet boundary condition on the enlarged domain $\hat{E}$} In this subsection, we construct a solution of a certain linear PDE on $\hat{E}$ with Dirichlet boundary condition. We again follow the arguments in Eichmair-Huang-Lee-Schoen \cite{Eichmair-Huang-Lee-Schoen} and Carlotto \cite{Carlotto}. Let us fix a nonnegative smooth function $\omega$ such that $\omega$ is supported in $E_0$ and $\omega = r^{-2}$ near infinity. By Hardy's inequality, we can find a positive constant $\kappa$ such that 
\begin{equation} 
\label{Hardy.inequality}
\int_{\hat{E}} \rho \, |df|^2 + \frac{1}{2} \int_{\hat{E}} \rho \, Q \, f^2 \geq \kappa \int_{\hat{E}} \rho \, \omega \, f^2 
\end{equation} 
for every smooth function $f$ on $\hat{E}$ that vanishes near infinity. Since the function 
\[\kappa \, \omega + R - 2 \, \Delta \log \rho - \frac{n+1}{n+2} \, |d\log \rho|^2\] 
is positive near infinity, we can find a large constant $\Lambda>4$ such that 
\begin{equation} 
\label{choice.of.Lambda}
(\Lambda-4) \, Q + \kappa \, \omega + R - 2 \, \Delta \log \rho - \frac{n+1}{n+2} \, |d\log \rho|^2 \geq 0 
\end{equation}
at each point on $\hat{E}$.

\begin{proposition}[Coercivity]
\label{coercivity}
Suppose that $f$ is a smooth function on $\hat{E}$ such that $f=0$ on $\partial \hat{E}$ and $f$ vanishes near infinity. Then 
\begin{align*} 
&\int_{\hat{E}} \rho \, |df|^2 + \frac{1}{2} \int_{\hat{E}} \rho \, \Big ( R - 2 \, \Delta \log \rho - \frac{n+1}{n+2} \, |d\log \rho|^2 - Q \Big ) \, f^2 \\
&\geq \frac{1}{2\Lambda} \int_{\hat{E}} \rho \, (Q + \kappa \, \omega) \, f^2. 
\end{align*}
\end{proposition}

\textbf{Proof.} 
We compute 
\begin{align*}
&\int_{\hat{E}} \rho \, |df|^2 + \frac{1}{2} \int_{\hat{E}} \rho \, \Big ( R - 2 \, \Delta \log \rho - \frac{n+1}{n+2} \, |d\log \rho|^2 - Q \Big ) \, f^2 \\ 
&- \frac{1}{2\Lambda} \int_{\hat{E}} \rho \, (Q + \kappa \, \omega) \, f^2 \\ 
&= \frac{\Lambda-1}{\Lambda} \, \bigg ( \int_{\hat{E}} \rho \, |df|^2 + \frac{1}{2} \int_{\hat{E}} \rho \, \Big ( R - 2 \, \Delta \log \rho - \frac{n+1}{n+2} \, |d\log \rho|^2 - 2Q \Big ) \, f^2 \bigg ) \\ 
&+ \frac{1}{\Lambda} \, \bigg ( \int_{\hat{E}} \rho \, |df|^2 + \frac{1}{2} \int_{\hat{E}} \rho \, Q \, f^2 - \kappa \int_{\hat{E}} \rho \, \omega \, f^2 \bigg ) \\ 
&+ \frac{1}{2\Lambda} \int_{\hat{E}} \rho \, \Big ( (\Lambda-4) \, Q + \kappa \, \omega + R - 2 \, \Delta \log \rho - \frac{n+1}{n+2} \, |d\log \rho|^2 \Big ) \, f^2. 
\end{align*} 
The first term on the right hand side is nonnegative since $(M,g,\rho,Q)$ is an $n$-dataset. The second term on the right hand side is nonnegative by the Hardy inequality (\ref{Hardy.inequality}). The third term on the right hand side is nonnegative in view of (\ref{choice.of.Lambda}). This completes the proof of Proposition \ref{coercivity}. \\

Let $\mathcal{H}$ denote the set of all functions $f \in H_{\text{\rm loc}}^1(\hat{E})$ with the property that 
\[\int_{\hat{E}} \rho \, |df|^2 + \int_{\hat{E}} \rho \, (Q+\omega) \, f^2 < \infty\] 
and the boundary trace of $f$ along $\partial \hat{E}$ vanishes. We define 
\[\|f\|_{\mathcal{H}}^2 = \int_{\hat{E}} \rho \, |df|^2 + \int_{\hat{E}} \rho \, (Q+\omega) \, f^2\] 
for all $f \in \mathcal{H}$. 

Let us fix a nonnegative smooth function $v_0$ on $\hat{E}$ such that $v_0=0$ on $\partial \hat{E}$ and $v_0=1$ near infinity.

\begin{proposition}
\label{minimizer.of.quadratic.form}
We can find a function $v \in \mathcal{H}$ with the property that $v$ minimizes the functional 
\begin{align*} 
&\int_{\hat{E}} \rho \, |dv|^2 + \frac{1}{2} \int_{\hat{E}} \rho \, \Big ( R - 2 \, \Delta \log \rho - \frac{n+1}{n+2} \, |d\log \rho|^2 - Q \Big ) \, v^2 \\ 
&+ 2 \int_{\hat{E}} \rho \, \langle dv_0,dv \rangle + \int_{\hat{E}} \rho \, \Big ( R - 2 \, \Delta \log \rho - \frac{n+1}{n+2} \, |d\log \rho|^2 - Q \Big ) \, v_0 \, v 
\end{align*}
among all functions $v \in \mathcal{H}$. Moreover, we can choose $v$ so that $v_0+v \geq 0$. 
\end{proposition}

\textbf{Proof.} 
It follows from Proposition \ref{coercivity} that 
\begin{align*} 
&\int_{\hat{E}} \rho \, |df|^2 + \frac{1}{2} \int_{\hat{E}} \rho \, \Big ( R - 2 \, \Delta \log \rho - \frac{n+1}{n+2} \, |d\log \rho|^2 - Q \Big ) \, f^2 \\
&\geq \frac{1}{2\Lambda} \int_{\hat{E}} \rho \, (Q + \kappa \, \omega) \, f^2. 
\end{align*} 
for all $f \in \mathcal{H}$. Using this coercivity property, the existence of a minimizer follows easily. By replacing $v$ by $|v_0+v|-v_0$, we can arrange that $v_0+v$ is nonnegative. This completes the proof of Proposition \ref{minimizer.of.quadratic.form}. \\

Let $v$ denote the minimizer constructed in Proposition \ref{minimizer.of.quadratic.form}. By elliptic regularity theory, $v$ is a smooth solution of the PDE
\begin{align} 
\label{pde.for.v}
&-\Delta (v_0+v) - \langle d\log \rho,d(v_0+v) \rangle \notag \\ 
&+ \frac{1}{2} \, \Big ( R - 2 \, \Delta \log \rho - \frac{n+1}{n+2} \, |d\log \rho|^2 - Q \Big ) \, (v_0+v) = 0 
\end{align}
on the domain $\hat{E}$ with Dirichlet boundary condition $v=0$ on $\partial \hat{E}$. Since $v \in \mathcal{H}$, the function $v_0+v$ does not vanish identically. Moreover, the function $v_0+v$ is nonnegative. Using the strict maximum principle, we conclude that the function $v_0+v$ is strictly positive everywhere. 

Our assumptions imply that 
\[\Big | \bar{D}^m \Big ( R - 2 \, \Delta \log \rho - \frac{n+1}{n+2} \, |d\log \rho|^2 - Q \Big ) \Big |_{\bar{g}} \leq O(r^{-n-m-2\delta})\] 
for every nonnegative integer $m$. Standard results for linear PDE \cite{Meyers} imply that there exists a real number $\gamma$ such that 
\[|\bar{D}^m(v - \gamma \, r^{2-n})|_{\bar{g}} \leq O(r^{2-n-m-2\hat{\delta}})\] 
for every nonnegative integer $m$, where $\hat{\delta} \in (0,\delta)$. 

\begin{proposition}
\label{leading.term.in.asymptotic.expansion.of.v}
We have $\gamma < 0$.
\end{proposition}

\textbf{Proof.} 
Since $(M,g,\rho,Q)$ is an $n$-dataset, we know that 
\[\int_{\hat{E}} \rho \, |df|^2 + \frac{1}{2} \int_{\hat{E}} \rho \, \Big ( R - 2 \, \Delta \log \rho - \frac{n+1}{n+2} \, |d \log \rho|^2 - 2Q \Big ) \, f^2 \geq 0\] 
for every smooth function $f$ on $\hat{E}$ with the property that $f=0$ on $\partial \hat{E}$ and $f=1$ near infinity. By approximation, the preceding inequality also holds for the function $f = v_0+v$. This gives 
\[\int_{\hat{E}} \rho \, |d(v_0+v)|^2 + \frac{1}{2} \int_{\hat{E}} \rho \, \Big ( R - 2 \, \Delta \log \rho - \frac{n+1}{n+2} \, |d \log \rho|^2 - 2Q \Big ) \, (v_0+v)^2 \geq 0.\] 
On the other hand, using (\ref{pde.for.v}), we obtain 
\begin{align*} 
&\int_{\hat{E}} \rho \, |d(v_0+v)|^2 + \frac{1}{2} \int_{\hat{E}} \rho \, \Big ( R - 2 \, \Delta \log \rho - \frac{n+1}{n+2} \, |d \log \rho|^2 - Q \Big ) \, (v_0+v)^2 \\ 
&= \int_{\hat{E}} \text{\rm div}(\rho \, (v_0+v) \, d(v_0+v)) = -(n-2) \, |S^{n-1}| \, \gamma. 
\end{align*}
Note that there is no boundary term on $\partial \hat{E}$ since $v_0+v$ vanishes on $\partial \hat{E}$. Finally, since the function $Q$ is strictly positive and the function $v_0+v$ is strictly positive, we know that 
\[\int_{\hat{E}} \rho \, Q \, (v_0+v)^2 > 0.\] 
Putting these facts together, we conclude that $-(n-2) \, |S^{n-1}| \, \gamma > 0$. This completes the proof of Proposition \ref{leading.term.in.asymptotic.expansion.of.v}. \\ 

\textbf{From now on, we will work exclusively on the closure of the domain $E$.} We define $\hat{v} = v_0+v$ and $\hat{\rho} = \rho \, \hat{v}$. With this understood, $\hat{v}$ and $\hat{\rho}$ are strictly positive smooth functions on the closure of $E$. The function $\hat{v}$ satisfies 
\[|\bar{D}^m(\hat{v} - (1 + \gamma \, r^{2-n}))|_{\bar{g}} \leq O(r^{2-n-m-2\hat{\delta}})\] 
for every nonnegative integer $m$. Hence, if we put $\hat{\beta} = \beta+\gamma$, then the function $\hat{\rho}$ satisfies 
\[|\bar{D}^m(\hat{\rho} - (1 + \hat{\beta} \, r^{2-n}))|_{\bar{g}} \leq O(r^{2-n-m-2\hat{\delta}})\] 
for every nonnegative integer $m$. Using Proposition \ref{leading.term.in.asymptotic.expansion.of.v} and the inequality (\ref{assumption}), we obtain 
\begin{equation}
\label{consequence.of.assumption}
(n-1)\alpha + 2\hat{\beta} < 0. 
\end{equation}
Finally, the identity (\ref{pde.for.v}) gives 
\begin{align} 
\label{pde.for.hat.v}
&-\Delta \hat{v} - \langle d\log \rho,d\hat{v} \rangle \notag \\ 
&+ \frac{1}{2} \, \Big ( R - 2 \, \Delta \log \rho - \frac{n+1}{n+2} \, |d\log \rho|^2 - Q \Big ) \, \hat{v} = 0 
\end{align} 
on the domain $E$. This implies 
\begin{align} 
\label{pde.for.log.hat.v}
&-\Delta \log \hat{v} - |d \log \hat{v}|^2 - \langle d\log \rho,d\log \hat{v} \rangle \notag \\ 
&+ \frac{1}{2} \, \Big ( R - 2 \, \Delta \log \rho - \frac{n+1}{n+2} \, |d\log \rho|^2 - Q \Big ) = 0 
\end{align}
on the domain $E$.

\subsection{A family of hypersurfaces in the asymptotically flat end $E_0$ with positive $\hat{\rho}$-weighted mean curvature} Throughout this subsection, we identify the asymptotically flat end $E_0$ with the complement of the unit ball in $\mathbb{R}^n$. For $\lambda>0$ sufficiently large, we define two hypersurfaces $N_\lambda^+$ and $N_\lambda^-$ by 
\[N_\lambda^+ = \Big \{ x_n = \lambda - (n-3+\hat{\delta})^{-1} \, (x_1^2+\hdots+x_{n-1}^2+\lambda^2)^{-\frac{n-3+\hat{\delta}}{2}} \Big \}\] 
and 
\[N_\lambda^- = \Big \{ -x_n = \lambda - (n-3+\hat{\delta})^{-1} \, (x_1^2+\hdots+x_{n-1}^2+\lambda^2)^{-\frac{n-3+\hat{\delta}}{2}} \Big \}.\] 
We choose the unit normal vector field along $N_\lambda^+$ so that $dx_n(\nu_{N_\lambda^+}) > 0$ at each point on $N_\lambda^+$. We choose the unit normal vector field along $N_\lambda^-$ so that $dx_n(\nu_{N_\lambda^-}) < 0$ at each point on $N_\lambda^-$. The following proposition is similar to the classical work of Schoen and Yau.

\begin{proposition}
\label{weighted.mean.curvature}
If $\lambda>0$ is sufficiently large, then the hypersurface $N_\lambda^+$ satisfies 
\[H_{N_\lambda^+} + \langle \nabla \log \hat{\rho},\nu_{N_\lambda^+} \rangle > 0\] 
and the hypersurface $N_\lambda^-$ satisfies 
\[H_{N_\lambda^-} + \langle \nabla \log \hat{\rho},\nu_{N_\lambda^-} \rangle > 0.\] 
Here, the mean curvature and unit normal vector are computed with respect to the metric $g$. 
\end{proposition}

\textbf{Proof.}
We only prove the assertion for $N_\lambda^+$. The proof for $N_\lambda^-$ is analogous. Let $\bar{H}_{N_\lambda^+}$ denote the mean curvature of the hypersurface $N_\lambda^+$ with respect to the Euclidean metric $\bar{g}$. Then 
\begin{align*} 
&\Big | \bar{H}_{N_\lambda^+} - (x_1^2+\hdots+x_{n-1}^2+\lambda^2)^{-\frac{n+1+\hat{\delta}}{2}} \, \big ( \hat{\delta} \, (x_1^2+\hdots+x_{n-1}^2) - (n-1) \, \lambda^2 \big ) \Big | \\ 
&\leq C \, r^{1-n-2\hat{\delta}}.
\end{align*} 
Moreover, the metric $g$ satisfies 
\[|g - (1+\alpha \, r^{2-n}) \, \bar{g}|_{\bar{g}} \leq C \, r^{2-n-2\hat{\delta}}\] 
and 
\[|\bar{D}(g - (1+\alpha \, r^{2-n}) \, \bar{g})|_{\bar{g}} \leq C \, r^{1-n-2\hat{\delta}}.\] 
Consequently, the mean curvature of $N_\lambda^+$ with respect to the metric $g$ satisfies 
\begin{align*} 
&\Big | H_{N_\lambda^+} + \frac{1}{2} \, (n-2)(n-1) \, \alpha \, \lambda \, r^{-n} \\ 
&- (x_1^2+\hdots+x_{n-1}^2+\lambda^2)^{-\frac{n+1+\hat{\delta}}{2}} \, \big ( \hat{\delta} \, (x_1^2+\hdots+x_{n-1}^2) - (n-1) \, \lambda^2 \big ) \Big | \\ 
&\leq C \, r^{1-n-2\hat{\delta}}.
\end{align*} 
The function $\hat{\rho}$ satisfies 
\[|\hat{\rho} - (1 + \hat{\beta} \, r^{2-n})| \leq C \, r^{2-n-2\hat{\delta}}\] 
and 
\[|\bar{D}(\hat{\rho} - (1 + \hat{\beta} \, r^{2-n}))|_{\bar{g}} \leq C \, r^{1-n-2\hat{\delta}}.\] 
Consequently, the normal derivative of the function $\log \hat{\rho}$ along $N_\lambda^+$ satisfies 
\[|\langle \nabla \log \hat{\rho},\nu_{N_\lambda^+} \rangle + (n-2) \, \hat{\beta} \, \lambda \, r^{-n}| \leq C \, r^{1-n-2\hat{\delta}}.\] 
Putting these facts together, we obtain 
\begin{align*} 
&\Big | H_{N_\lambda^+} + \langle \nabla \log \hat{\rho},\nu_{N_\lambda^+} \rangle + \frac{1}{2} \, (n-2) \, ((n-1)\alpha + 2\hat{\beta}) \, \lambda \, r^{-n} \\ 
&- (x_1^2+\hdots+x_{n-1}^2+\lambda^2)^{-\frac{n+1+\hat{\delta}}{2}} \, \big ( \hat{\delta} \, (x_1^2+\hdots+x_{n-1}^2) - (n-1) \, \lambda^2 \big ) \Big | \\ 
&\leq C \, r^{1-n-2\hat{\delta}}. 
\end{align*}
Recall that $(n-1)\alpha + 2\hat{\beta} < 0$ by (\ref{consequence.of.assumption}). 

It is convenient to divide $N_\lambda^+$ into two regions. In the region 
\[N_\lambda^+ \cap \{\hat{\delta} \, (x_1^2+\hdots+x_{n-1}^2) \geq 4(n-1)\lambda^2\},\] 
we have 
\begin{align*} 
&H_{N_\lambda^+} + \langle \nabla \log \hat{\rho},\nu_{N_\lambda^+} \rangle \\ 
&\geq (x_1^2+\hdots+x_{n-1}^2+\lambda^2)^{-\frac{n+1+\hat{\delta}}{2}} \, \big ( \hat{\delta} \, (x_1^2+\hdots+x_{n-1}^2) - (n-1) \, \lambda^2 \big ) \\ 
&- C \, r^{1-n-2\hat{\delta}} \\ 
&\geq (x_1^2+\hdots+x_{n-1}^2+\lambda^2)^{-\frac{n+1+\hat{\delta}}{2}} \, \Big ( \frac{1}{2} \, \hat{\delta} \, (x_1^2+\hdots+x_{n-1}^2) + (n-1) \, \lambda^2 \Big ) \\ 
&- C \, r^{1-n-2\hat{\delta}}, 
\end{align*} 
and the expression on the right hand side is positive if $\lambda$ is sufficiently large. Finally, in the region 
\[N_\lambda^+ \cap \{\hat{\delta} \, (x_1^2+\hdots+x_{n-1}^2) \leq 4(n-1)\lambda^2\},\] 
we have 
\[H_{N_\lambda^+} + \langle \nabla \log \hat{\rho},\nu_{N_\lambda^+} \rangle \geq -\frac{1}{2} \, (n-2) \, ((n-1)\alpha + 2\hat{\beta}) \, \lambda \, r^{-n} - C \, r^{1-n-\hat{\delta}},\] 
and the expression on the right hand side is positive if $\lambda$ is sufficiently large. This completes the proof of Proposition \ref{weighted.mean.curvature}. \\

Let us fix a large constant $\lambda_0$ with the following properties: 
\begin{itemize}
\item For each $\lambda \geq \lambda_0$, we have 
\[H_{N_\lambda^+} + \langle \nabla \log \hat{\rho},\nu_{N_\lambda^+} \rangle > 0\] 
at each point on $N_\lambda^+$.
\item For each $\lambda \geq \lambda_0$, we have 
\[H_{N_\lambda^-} + \langle \nabla \log \hat{\rho},\nu_{N_\lambda^-} \rangle > 0\] 
at each point on $N_\lambda^-$.
\item The hypersurfaces $\{N_\lambda^+: \lambda \in [\lambda_0,4\lambda_0]\}$ form a foliation.
\item The hypersurfaces $\{N_\lambda^-: \lambda \in [\lambda_0,4\lambda_0]\}$ form a foliation.
\end{itemize}

\begin{proposition}
\label{angle}
There exists a large constant $s_0$ (depending on $\lambda_0$) with the following significance. If $s \geq s_0$ and $\lambda \in [\lambda_0,4\lambda_0]$, then the hypersurface 
\[W_s = \{x_1^2+\hdots+x_{n-1}^2=s^2\}\] 
intersects $N_\lambda^+$ and $N_\lambda^-$ transversally. Moreover, $\langle \nu_{N_\lambda^+},\nu_{W_s} \rangle < 0$ at each point on $N_\lambda^+ \cap W_s$ and $\langle \nu_{N_\lambda^-},\nu_{W_s} \rangle < 0$ at each point on $N_\lambda^- \cap W_s$. Here, $\nu_{W_s}$ denotes the outward-pointing unit normal vector to $W_s$.
\end{proposition}

\textbf{Proof.} 
We only prove the assertion for $N_\lambda^+$. The proof for $N_\lambda^-$ is analogous. Let $\bar{\nu}_{N_\lambda^+}$ denote the unit normal to $N_\lambda^+$ with respect to the Euclidean metric $\bar{g}$, and let $\bar{\nu}_{W_s}$ denote the outward-pointing unit normal vector to $W_s$ with respect to the Euclidean metric. We compute 
\[\langle \bar{\nu}_{N_\lambda^+},\bar{\nu}_{W_s} \rangle_{\bar{g}} = -s \, (s^2+\lambda^2)^{-\frac{n-1+\hat{\delta}}{2}} \, \big (1 + s^2 \, (s^2+\lambda^2)^{1-n-\hat{\delta}} \big )^{-\frac{1}{2}}\] 
at each point on $N_\lambda^+ \cap W_s$. Hence, if $s$ is sufficiently large (depending on $\lambda_0$), then 
\[|\langle \bar{\nu}_{N_\lambda^+},\bar{\nu}_{W_s} \rangle_{\bar{g}} + s^{2-n-\hat{\delta}}| \leq C \, s^{-n-\hat{\delta}}\] 
at each point on $N_\lambda^+ \cap W_s$. Since the metric $g$ satisfies 
\[|g - (1+\alpha \, r^{2-n}) \, \bar{g}|_{\bar{g}} \leq C \, r^{2-n-2\hat{\delta}}\] 
and the Riemannian metric $(1+\alpha \, r^{2-n}) \, \bar{g}$ is conformal to $\bar{g}$, we conclude that 
\[|\langle \nu_{N_\lambda^+},\nu_{W_s} \rangle_g - \langle \bar{\nu}_{N_\lambda^+},\bar{\nu}_{W_s} \rangle_{\bar{g}}| \leq C \, s^{2-n-2\hat{\delta}}\] 
at each point on $N_\lambda^+ \cap W_s$. Therefore, if $s$ is sufficiently large (depending on $\lambda_0$), then we obtain 
\[|\langle \nu_{N_\lambda^+},\nu_{W_s} \rangle_g + s^{2-n-\hat{\delta}}| \leq C \, s^{2-n-2\hat{\delta}}\] 
at each point on $N_\lambda^+ \cap W_s$. This completes the proof of Proposition \ref{angle}. \\

\begin{definition}
For each $\lambda \in [\lambda_0,4\lambda_0]$, we define an open domain $E_{\text{\rm slab},\lambda}$ by 
\[E_{\text{\rm slab},\lambda} = E \setminus \Big \{ |x_n| \geq \lambda - (n-3+\hat{\delta})^{-1} \, (x_1^2+\hdots+x_{n-1}^2+\lambda^2)^{-\frac{n-3+\hat{\delta}}{2}} \Big \}.\] 
Note that $\partial E_{\text{\rm slab},\lambda} = N_\lambda^+ \cup N_\lambda^- \cup \partial E$. 
\end{definition}

We define a vector field $X$ on $E_{\text{\rm slab},4\lambda_0} \setminus E_{\text{\rm slab},\lambda_0}$ so that $X$ is the unit normal vector field to the foliation $N_\lambda^+$ in the region 
\begin{align*} 
&\Big \{ \lambda_0 - (n-3+\hat{\delta})^{-1} \, (x_1^2+\hdots+x_{n-1}^2+\lambda_0^2)^{-\frac{n-3+\hat{\delta}}{2}} \\ 
&\hspace{4mm} \leq x_n < 4\lambda_0 - (n-3+\hat{\delta})^{-1} \, (x_1^2+\hdots+x_{n-1}^2+16\lambda_0^2)^{-\frac{n-3+\hat{\delta}}{2}} \Big \} 
\end{align*} 
and $X$ is the unit normal vector field to the foliation $N_\lambda^-$ in the region 
\begin{align*} 
&\Big \{ \lambda_0 - (n-3+\hat{\delta})^{-1} \, (x_1^2+\hdots+x_{n-1}^2+\lambda_0^2)^{-\frac{n-3+\hat{\delta}}{2}} \\ 
&\hspace{4mm} \leq -x_n < 4\lambda_0 - (n-3+\hat{\delta})^{-1} \, (x_1^2+\hdots+x_{n-1}^2+16\lambda_0^2)^{-\frac{n-3+\hat{\delta}}{2}} \Big \}. 
\end{align*} 
Note that $|X|=1$ at each point in $E_{\text{\rm slab},4\lambda_0} \setminus E_{\text{\rm slab},\lambda_0}$. 

\begin{proposition}
\label{X.is.a.calibration}
Let $X$ denote the vector field constructed above. Then $\text{\rm div}_g(\hat{\rho} \, X) > 0$ at each point in $E_{\text{\rm slab},4\lambda_0} \setminus E_{\text{\rm slab},\lambda_0}$. Moreover, if $s \geq s_0$, then $\langle X,\nu_{W_s} \rangle < 0$ at each point in $(E_{\text{\rm slab},4\lambda_0} \setminus E_{\text{\rm slab},\lambda_0}) \cap W_s$.
\end{proposition}

\textbf{Proof.} 
The first statement follows from Proposition \ref{weighted.mean.curvature}. The second statement follows from Proposition \ref{angle}. \\

\subsection{Construction of a $\mu$-bubble (possibly with singularities)} In this subsection, we construct a suitable $\mu$-bubble. The use of $\mu$-bubbles in the study of scalar curvature was pioneered by Gromov \cite{Gromov}. 

In the following, $\Phi$ will denote the function constructed in Lemma \ref{Phi}. 

\begin{lemma} 
\label{Y.is.a.calibration}
We can find a small positive constant $\eta_0$ so that the following statement holds. The distance function $d_{(M,g)}(\cdot,\partial E)$ is smooth on the tubular neighborhood $E \cap \mathcal{N}_{(M,g)}(\partial E,4\eta_0)$. Moreover, we have $\text{\rm div}(\hat{\rho} \, Y) > \hat{\rho} \, \Phi$ at each point in $E \cap \mathcal{N}_{(M,g)}(\partial E,4\eta_0)$, where $Y$ denotes the gradient of the function $-d_{(M,g)}(\cdot,\partial E)$. 
\end{lemma}

\textbf{Proof.} 
By Lemma \ref{Phi}, we know that $\Phi \to -\infty$ on the boundary $\partial E$. From this, the assertion follows. This completes the proof of Lemma \ref{Y.is.a.calibration}. \\

Let us consider a sequence $s_j \to \infty$. For each $\lambda \in [\lambda_0,4\lambda_0]$ and each $j$, we define an open domain $E_{\text{\rm slab},\lambda}^{(j)}$ by 
\[E_{\text{\rm slab},\lambda}^{(j)} = E_{\text{\rm slab},\lambda} \setminus \{x_1^2+\hdots+x_{n-1}^2 \geq s_j^2\}.\] 
For each $j$, we consider a variational problem on the domain $E_{\text{\rm slab},4\lambda_0}^{(j)}$. 

\begin{definition}
Let $\eta_0$ be chosen as in Lemma \ref{Y.is.a.calibration}. We denote by $\mathcal{Z}^{(j)}$ the collection of Caccioppoli sets $\Omega \subset E_{\text{\rm slab},4\lambda_0}^{(j)}$ with the property that 
\[\Omega \cap \mathcal{N}_{(M,g)}(\partial E,\eta_0) = \emptyset,\] 
\[\Omega \cap \bigcup_{\lambda \in [3\lambda_0,4\lambda_0)} N_\lambda^+ = \emptyset,\] 
and 
\[\bigcup_{\lambda \in [3\lambda_0,4\lambda_0)} N_\lambda^- \setminus \{x_1^2+\hdots+x_{n-1}^2 \geq s_j^2\} \subset \Omega.\] 
Note that the domain $E_{\text{\rm slab},4\lambda_0}^{(j)} \cap \{x_n < -10\}$ belongs to $\mathcal{Z}^{(j)}$. In particular, $\mathcal{Z}^{(j)} \neq \emptyset$.
\end{definition}

\begin{definition}
\label{definition.of.functional}
For each $\Omega \in \mathcal{Z}^{(j)}$, we define 
\[\mathcal{F}^{(j)}(\Omega) = \int_{\partial^* \Omega \cap E_{\text{\rm slab},4\lambda_0}^{(j)}} \hat{\rho} \, d\mathcal{H}^{n-1} - \int_\Omega \hat{\rho} \, \Phi \, d\mathcal{H}^n.\] 
Since $\Phi \leq 0$ at each point in $E$, we have 
\[\mathcal{F}^{(j)}(\Omega) \geq \int_{\partial^* \Omega \cap E_{\text{\rm slab},4\lambda_0}^{(j)}} \hat{\rho} \, d\mathcal{H}^{n-1}\] 
for each $\Omega \in \mathcal{Z}^{(j)}$.
\end{definition}

\begin{lemma} 
\label{modification.of.Omega.1}
Let $\eta_0$ be chosen as in Lemma \ref{Y.is.a.calibration}. Let $\Omega \in \mathcal{Z}^{(j)}$, and suppose that $\Omega$ has smooth boundary in $E_{\text{\rm slab},4\lambda_0}^{(j)}$. Let us choose $\tilde{\eta} \in (3\eta_0,4\eta_0)$ so that $\partial \Omega$ is transversal to $\partial \mathcal{N}_{(M,g)}(\partial E,\tilde{\eta})$. Then $\mathcal{F}^{(j)}(\tilde{\Omega}) \leq \mathcal{F}^{(j)}(\Omega)$, where $\tilde{\Omega} = \Omega \setminus \mathcal{N}_{(M,g)}(\partial E,\tilde{\eta})$.
\end{lemma}

\textbf{Proof.} 
Let $Y$ denote the vector field constructed in Lemma \ref{Y.is.a.calibration}. Then $|Y|=1$ and $\text{\rm div}_g(\hat{\rho} \, Y) > \hat{\rho} \, \Phi$ at each point in $E \cap \mathcal{N}_{(M,g)}(\partial E,4\eta_0)$. Integrating this inequality over $\Omega \cap \mathcal{N}_{(M,g)}(\partial E,\tilde{\eta}) \subset E \cap \mathcal{N}_{(M,g)}(\partial E,4\eta_0)$ gives 
\begin{align*} 
&\int_{\partial \Omega \cap \mathcal{N}_{(M,g)}(\partial E,\tilde{\eta})} \hat{\rho} \, d\mathcal{H}^{n-1} - \int_{\partial \Omega \cap \partial \mathcal{N}_{(M,g)}(\partial E,\tilde{\eta})} \hat{\rho} \, d\mathcal{H}^{n-1} \\ 
&\geq \int_{\partial \Omega \cap \mathcal{N}_{(M,g)}(\partial E,\tilde{\eta})} \hat{\rho} \, \langle Y,\nu_{\partial \Omega} \rangle \, d\mathcal{H}^{n-1} - \int_{\partial \Omega \cap \partial \mathcal{N}_{(M,g)}(\partial E,\tilde{\eta})} \hat{\rho} \, d\mathcal{H}^{n-1} \\ 
&= \int_{\Omega \cap \mathcal{N}_{(M,g)}(\partial E,\tilde{\eta})} \text{\rm div}(\hat{\rho} \, Y) \, d\mathcal{H}^n \\ 
&\geq \int_{\Omega \cap \mathcal{N}_{(M,g)}(\partial E,\tilde{\eta})} \hat{\rho} \, \Phi \, d\mathcal{H}^n. 
\end{align*} 
This implies $\mathcal{F}^{(j)}(\tilde{\Omega}) \leq \mathcal{F}^{(j)}(\Omega)$. This completes the proof of Lemma \ref{modification.of.Omega.1}. \\

\begin{lemma} 
\label{modification.of.Omega.2}
Let $\Omega \in \mathcal{Z}^{(j)}$, and suppose that $\Omega$ has smooth boundary in $E_{\text{\rm slab},4\lambda_0}^{(j)}$. Let us choose $\tilde{\lambda} \in (\lambda_0,\frac{3\lambda_0}{2})$ so that $\partial \Omega$ is transversal to $N_{\tilde{\lambda}}^+$. Then $\mathcal{F}^{(j)}(\tilde{\Omega}) \leq \mathcal{F}^{(j)}(\Omega)$, where 
\[\tilde{\Omega} 
= \Omega \setminus \bigcup_{\lambda \in (\tilde{\lambda},4\lambda_0)} N_\lambda^+.\]
\end{lemma}

\textbf{Proof.} 
Let $X$ denote the vector field constructed in the previous subsection. Then $|X| = 1$, $\langle X,\nu_{W_{s_j}} \rangle < 0$ on $(E_{\text{\rm slab},4\lambda_0} \setminus E_{\text{\rm slab},\lambda_0}) \cap W_{s_j}$, and $\text{\rm div}_g(\hat{\rho} \, X) > 0$ in $E_{\text{\rm slab},4\lambda_0} \setminus E_{\text{\rm slab},\lambda_0}$. We now integrate the inequality $\text{\rm div}_g(\hat{\rho} \, X) > 0$ over 
\[\Omega \cap \bigcup_{\lambda \in (\tilde{\lambda},4\lambda_0)} N_\lambda^+.\] Consequently, the $\hat{\rho}$-weighted area of 
\[\partial \Omega \cap \bigcup_{\lambda \in (\tilde{\lambda},4\lambda_0)} N_\lambda^+\] 
is bounded from below by the $\hat{\rho}$-weighted area of $\Omega \cap N_{\tilde{\lambda}}^+$. Since $\Phi=0$ on $E_0$, it follows that $\mathcal{F}^{(j)}(\tilde{\Omega}) \leq \mathcal{F}^{(j)}(\Omega)$. This completes the proof of Lemma \ref{modification.of.Omega.2}. \\

\begin{lemma} 
\label{modification.of.Omega.3}
Let $\Omega \in \mathcal{Z}^{(j)}$, and suppose that $\Omega$ has smooth boundary in $E_{\text{\rm slab},4\lambda_0}^{(j)}$. Let us choose $\tilde{\lambda} \in (\lambda_0,\frac{3\lambda_0}{2})$ so that $\partial \Omega$ is transversal to $N_{\tilde{\lambda}}^-$. Then $\mathcal{F}^{(j)}(\tilde{\Omega}) \leq \mathcal{F}^{(j)}(\Omega)$, where 
\[\tilde{\Omega} = \Omega \cup \bigg ( \bigcup_{\lambda \in (\tilde{\lambda},4\lambda_0)} N_\lambda^- \setminus \{x_1^2+\hdots+x_{n-1}^2 \geq s_j^2\} \bigg ).\]
\end{lemma}

\textbf{Proof.} 
Let $X$ denote the vector field constructed in the previous subsection. Then $|X| = 1$, $\langle X,\nu_{W_{s_j}} \rangle < 0$ on $(E_{\text{\rm slab},4\lambda_0} \setminus E_{\text{\rm slab},\lambda_0}) \cap W_{s_j}$, and $\text{\rm div}_g(\hat{\rho} \, X) > 0$ in $E_{\text{\rm slab},4\lambda_0} \setminus E_{\text{\rm slab},\lambda_0}$. We now integrate the inequality $\text{\rm div}_g(\hat{\rho} \, X) > 0$ over 
\[\bigcup_{\lambda \in (\tilde{\lambda},4\lambda_0)} N_\lambda^- \setminus (\Omega \cup \{x_1^2+\hdots+x_{n-1}^2 \geq s_j^2\}).\] 
Consequently, the $\hat{\rho}$-weighted area of 
\[\partial \Omega \cap \bigcup_{\lambda \in (\tilde{\lambda},4\lambda_0)} N_\lambda^-\] 
is bounded from below by the $\hat{\rho}$-weighted area of $N_{\tilde{\lambda}}^- \setminus \Omega$. Since $\Phi=0$ on $E_0$, it follows that $\mathcal{F}^{(j)}(\tilde{\Omega}) \leq \mathcal{F}^{(j)}(\Omega)$. This completes the proof of Lemma \ref{modification.of.Omega.3}. \\

\begin{proposition}[Existence of a minimizer]
\label{existence.of.minimizer}
For each $j$, there exists a Caccioppoli set $\hat{\Omega}^{(j)} \in \mathcal{Z}^{(j)}$ which minimizes the functional $\mathcal{F}^{(j)}$. Moreover, 
\begin{equation} 
\label{a}
\hat{\Omega}^{(j)} \cap \mathcal{N}_{(M,g)}(\partial E,2\eta_0) = \emptyset, 
\end{equation}
\begin{equation} 
\label{b}
\hat{\Omega}^{(j)} \cap \bigcup_{\lambda \in [2\lambda_0,4\lambda_0)} N_\lambda^+ = \emptyset, 
\end{equation}
and 
\begin{equation} 
\label{c}
\bigcup_{\lambda \in [2\lambda_0,4\lambda_0)} N_\lambda^- \setminus \{x_1^2+\hdots+x_{n-1}^2 \geq s_j^2\} \subset \hat{\Omega}^{(j)}. 
\end{equation}
In particular, the reduced boundary of $\hat{\Omega}^{(j)}$ is contained in the closure of $E_{\text{\rm slab},2\lambda_0}^{(j)} \setminus \mathcal{N}_{(M,g)}(\partial E,2\eta_0)$.
\end{proposition}

\textbf{Proof.} 
We fix an integer $j$. Let $\{\Omega^{(j,l)}: l=1,2,\hdots\} \subset \mathcal{Z}^{(j)}$ be a minimizing sequence for the function $\mathcal{F}^{(j)}$. By Theorem 13.8 in \cite{Maggi}, we may assume that $\Omega^{(j,l)}$ has smooth boundary in $E_{\text{\rm slab},4\lambda_0}^{(j)}$. In view of Lemma \ref{modification.of.Omega.1}, Lemma \ref{modification.of.Omega.2}, and Lemma \ref{modification.of.Omega.3}, we may further assume that 
\begin{equation} 
\label{d}
\Omega^{(j,l)} \cap \mathcal{N}_{(M,g)}(\partial E,3\eta_0) = \emptyset, 
\end{equation}
\begin{equation} 
\label{e}
\Omega^{(j,l)} \cap \bigcup_{\lambda \in [\frac{3\lambda_0}{2},4\lambda_0)} N_\lambda^+ = \emptyset, 
\end{equation}
and 
\begin{equation} 
\label{f}
\bigcup_{\lambda \in [\frac{3\lambda_0}{2},4\lambda_0)} N_\lambda^- \setminus \{x_1^2+\hdots+x_{n-1}^2 \geq s_j^2\} \subset \Omega^{(j,l)}. 
\end{equation}
Clearly, 
\[\sup_l \mathcal{H}^{n-1}(\partial^* \Omega^{(j,l)} \cap E_{\text{\rm slab},4\lambda_0}^{(j)}) < \infty.\] 
Thus, for each $j$, the sets $\{\Omega^{(j,l)}: l=1,2,\hdots\}$ have bounded perimeter. We now invoke the compactness theorem for BV functions. Hence, we can find a Caccioppoli set $\hat{\Omega}^{(j)}$ such that, after passing to a subsequence if necessary, $1_{\Omega^{(j,l)}}$ converges to $1_{\hat{\Omega}^{(j)}}$ strongly in $L^1$. The statements (\ref{a}), (\ref{b}), (\ref{c}) follow from (\ref{d}), (\ref{e}), (\ref{f}). Therefore, $\hat{\Omega}^{(j)} \in \mathcal{Z}^{(j)}$. Since the BV norm is lower semicontinuous with respect to $L^1$ convergence, it follows that 
\[\mathcal{F}^{(j)}(\hat{\Omega}^{(j)}) \leq \limsup_{l \to \infty} \mathcal{F}^{(j)}(\Omega^{(j,l)}).\] 
Thus, $\hat{\Omega}^{(j)}$ is the desired minimizer. This completes the proof of Proposition \ref{existence.of.minimizer}. \\

\begin{lemma}
\label{area.bound.for.hat.Omega_j}
Let $U$ be an open domain with smooth boundary which is contained in a compact subset of $E_{\text{\rm slab},3\lambda_0} \setminus \mathcal{N}_{(M,g)}(\partial E,\frac{3\eta_0}{2})$. Then 
\[\int_{\partial^* \hat{\Omega}^{(j)} \cap U} \hat{\rho} \, d\mathcal{H}^{n-1} \leq \int_{\partial U} \hat{\rho} \, d\mathcal{H}^{n-1}.\] 
\end{lemma}

\textbf{Proof.} 
Since $U$ is contained in a compact subset of $E_{\text{\rm slab},3\lambda_0} \setminus \mathcal{N}_{(M,g)}(\partial E,\frac{3\eta_0}{2})$, it follows that $\hat{\Omega}^{(j)} \setminus U \in \mathcal{Z}^{(j)}$. Using the minimization property of $\mathcal{F}^{(j)}$, we obtain 
\[\mathcal{F}^{(j)}(\hat{\Omega}^{(j)}) \leq \mathcal{F}^{(j)}(\hat{\Omega}^{(j)} \setminus U).\] 
This implies 
\[\int_{\partial^* \hat{\Omega}^{(j)} \cap E_{\text{\rm slab},4\lambda_0}^{(j)} \cap U} \hat{\rho} \, d\mathcal{H}^{n-1} - \int_{\hat{\Omega}^{(j)} \cap U} \hat{\rho} \, \Phi \, d\mathcal{H}^n \leq \int_{\partial U} \hat{\rho} \, d\mathcal{H}^{n-1}.\] 
Since $\Phi \leq 0$, the assertion follows. This completes the proof of Lemma \ref{area.bound.for.hat.Omega_j}. \\

\begin{lemma}
\label{area.growth.of.hat.Omega_j}
We can find a large number $s_*$ with the following significance. Suppose that $j$ is sufficiently large so that $s_j \geq s_*$. Then 
\[\int_{\partial^* \hat{\Omega}^{(j)} \cap E_{\text{\rm slab},4\lambda_0}^{(j)} \cap \{\frac{1}{4} \, \bar{s}^2 < x_1^2+\hdots+x_{n-1}^2 < \bar{s}^2\}} \hat{\rho} \, d\mathcal{H}^{n-1} \leq C \, \bar{s}^{n-1}\]
for each $\bar{s} \in [s_*,s_j]$, where $C$ is a uniform constant. 
\end{lemma}

\textbf{Proof.} 
Let 
\[U = E_{\text{\rm slab},\frac{5\lambda_0}{2}} \cap \Big \{ \frac{1}{4} \, \bar{s}^2 < x_1^2+\hdots+x_{n-1}^2 < \bar{s}^2 \Big \}.\] 
Clearly, $U$ is an open set which is contained in a compact subset of $E_{\text{\rm slab},3\lambda_0} \setminus \mathcal{N}_{(M,g)}(\partial E,\frac{3\eta_0}{2})$. Arguing as in the proof of Lemma \ref{area.bound.for.hat.Omega_j}, we obtain 
\[\int_{\partial^* \hat{\Omega}^{(j)} \cap U} \hat{\rho} \, d\mathcal{H}^{n-1} \leq \int_{\partial U} \hat{\rho} \, d\mathcal{H}^{n-1}.\] 
From this, the assertion follows. This completes the proof of Lemma \ref{area.growth.of.hat.Omega_j}. \\

In the next step, we take a limit of the minimizers $\hat{\Omega}^{(j)}$ as $j \to \infty$. To do that, we use ideas from the work of Eichmair and K\"orber \cite{Eichmair-Koerber}. Using Lemma \ref{area.bound.for.hat.Omega_j}, we obtain local area bounds for $\hat{\Omega}^{(j)}$. We again invoke the compactness theorem for BV functions. After passing to a subsequence, we can find a Caccioppoli set $\hat{\Omega}$ such that $1_{\hat{\Omega}^{(j)}} \to 1_{\hat{\Omega}}$ in $L_{\text{\rm loc}}^1(E_{\text{\rm slab},4\lambda_0})$.

\begin{definition}
Let $\hat{U}$ be an open set which is contained in a compact subset of $E_{\text{\rm slab},3\lambda_0} \setminus \mathcal{N}_{(M,g)}(\partial E,\frac{3\eta_0}{2})$. Given a Caccioppoli set $\Omega \subset E_{\text{\rm slab},4\lambda_0}$, we define 
\[\mathcal{F}(\Omega;\hat{U}) = \int_{\partial^* \Omega \cap \hat{U}} \hat{\rho} \, d\mathcal{H}^{n-1} - \int_{\Omega \cap \hat{U}} \hat{\rho} \, \Phi \, d\mathcal{H}^n.\] 
Since $\Phi \leq 0$ at each point in $E$, we obtain 
\[\mathcal{F}(\Omega;\hat{U}) \geq \int_{\partial^* \Omega \cap \hat{U}} \hat{\rho} \, d\mathcal{H}^{n-1}.\] 
\end{definition}

\begin{lemma}
\label{hat.Omega.is.a.minimizer}
Let $\hat{U}$ be an open domain with smooth boundary which is contained in a compact subset of $E_{\text{\rm slab},3\lambda_0} \setminus \mathcal{N}_{(M,g)}(\partial E,\frac{3\eta_0}{2})$. If $\Omega \subset E_{\text{\rm slab},4\lambda_0}$ is a Caccioppoli set with the property that the symmetric difference $\Omega \triangle \hat{\Omega}$ is contained in a compact subset of $\hat{U}$, then $\mathcal{F}(\hat{\Omega};\hat{U}) \leq \mathcal{F}(\Omega;\hat{U})$. 
\end{lemma}

\textbf{Proof.} 
For $s>0$ sufficiently small, we define $U_s = \{x \in \hat{U}: d_{(M,g)}(x,\partial \hat{U}) > s\}$. Recall that 
\[\int_{\hat{U}} |1_{\hat{\Omega}^{(j)}} - 1_{\hat{\Omega}}| \to 0\]
as $j \to \infty$. By the co-area formula, we can find a sequence of positive real numbers $\varepsilon_j \to 0$ and a sequence of positive real numbers $\hat{s}_j \to 0$ with the property that the boundary trace of $1_{\hat{\Omega}^{(j)}} - 1_{\hat{\Omega}}$ along $\partial U_{\hat{s}_j}$ has $L^1(\partial U_{\hat{s}_j})$-norm less than $\varepsilon_j$. 

By assumption, the symmetric difference $\Omega \triangle \hat{\Omega}$ is contained in a compact subset of $\hat{U}$. Hence, if $j$ is sufficiently large, then the boundary trace of $1_{\hat{\Omega}^{(j)}} - 1_\Omega$ along $\partial U_{\hat{s}_j}$ has $L^1(\partial U_{\hat{s}_j})$-norm less than $\varepsilon_j$. 

If $j$ is sufficiently large, then the set $(\hat{\Omega}^{(j)} \setminus U_{\hat{s}_j}) \cup (\Omega \cap U_{\hat{s}_j})$ belongs to $\mathcal{Z}^{(j)}$. Using the minimization property of $\mathcal{F}^{(j)}$, we obtain 
\[\mathcal{F}^{(j)}(\hat{\Omega}^{(j)}) \leq \mathcal{F}^{(j)} \big ( (\hat{\Omega}^{(j)} \setminus U_{\hat{s}_j}) \cup (\Omega \cap U_{\hat{s}_j}) \big )\] 
if $j$ is sufficiently large. This implies 
\[\mathcal{F}(\hat{\Omega}^{(j)};U_{\hat{s}_j}) \leq \mathcal{F}(\Omega;U_{\hat{s}_j}) + C\varepsilon_j\] 
if $j$ is sufficiently large. Finally, we send $j \to \infty$. Since the BV norm is lower semicontinuous with respect to $L^1$ convergence, we conclude that 
\[\mathcal{F}(\hat{\Omega};\tilde{U}) \leq \limsup_{j \to \infty} \mathcal{F}(\hat{\Omega}^{(j)};U_{\hat{s}_j}) \leq \limsup_{j \to \infty} (\mathcal{F}(\Omega;U_{\hat{s}_j}) + C\varepsilon_j) \leq \mathcal{F}(\Omega;\hat{U})\] 
for every open domain $\tilde{U}$ with the property that the closure of $\tilde{U}$ is contained in a compact subset of $\hat{U}$. This finally implies 
\[\mathcal{F}(\hat{\Omega};\hat{U}) \leq \mathcal{F}(\Omega;\hat{U}).\] 
This completes the proof of Lemma \ref{hat.Omega.is.a.minimizer}. \\

\begin{lemma}
\label{density.is.close.to.1}
If $\bar{s}$ is sufficiently large, then the following holds. If $(\bar{x}_1,\hdots,\bar{x}_{n-1})$ is a point in $\mathbb{R}^{n-1}$ with $\sqrt{\bar{x}_1^2+\hdots+\bar{x}_{n-1}^2} = \bar{s}$, then 
\begin{align*} 
&\int_{\partial^* \hat{\Omega} \cap E_{\text{\rm slab},4\lambda_0} \cap \{(x_1-\bar{x}_1)^2+\hdots+(x_{n-1}-\bar{x}_{n-1})^2 < \frac{1}{4} \, \bar{s}^2\}} \hat{\rho} \, d\mathcal{H}^{n-1} \\ 
&\leq (1+o(1)) \, |B^{n-1}| \, \Big ( \frac{1}{2} \, \bar{s} \Big )^{n-1}. 
\end{align*}
\end{lemma}

\textbf{Proof.} 
Let 
\[U = E_{\text{\rm slab},\frac{5\lambda_0}{2}} \cap \Big \{ (x_1-\bar{x}_1)^2+\hdots+(x_{n-1}-\bar{x}_{n-1})^2 < \frac{1}{4} \, \bar{s}^2 \Big \}.\] 
We can find an open domain $\hat{U}$ with smooth boundary such that $U$ is contained in a compact subset of $\hat{U}$ and $\hat{U}$ is contained in a compact subset of $E_{\text{\rm slab},3\lambda_0} \setminus \mathcal{N}_{(M,g)}(\partial E,\frac{3\eta_0}{2})$. Using Lemma \ref{hat.Omega.is.a.minimizer}, we obtain 
\[\mathcal{F}(\hat{\Omega};\hat{U}) \leq \mathcal{F}(\hat{\Omega} \setminus U;\hat{U}).\] 
This implies 
\[\int_{\partial^* \hat{\Omega} \cap U} \hat{\rho} \, d\mathcal{H}^{n-1} \leq \int_{\partial U \setminus N_{3\lambda_0}^+} \hat{\rho} \, d\mathcal{H}^{n-1} \leq (1+o(1)) \, |B^{n-1}| \, \Big ( \frac{1}{2} \, \bar{s} \Big )^{n-1}.\] 
This completes the proof of Lemma \ref{density.is.close.to.1}. \\

\begin{definition}
We denote by $\mathcal{S}$ the singular set of $\hat{\Omega}$, and by $\Sigma = (\partial^* \hat{\Omega} \cap E_{\text{\rm slab},4\lambda_0}) \setminus \mathcal{S}$ the regular part of $\hat{\Omega}$. Note that $\Sigma$ is a smooth (possibly disconnected) hypersurface in $E \setminus \mathcal{S}$, and $H_\Sigma + \langle \nabla \log \hat{\rho},\nu_\Sigma \rangle = \Phi$ at each point on $\Sigma$. 
\end{definition}

In the remainder of this subsection, we show that the singular set $\mathcal{S}$ is compact and $\Sigma$ is asymptotically planar near infinity. 

\begin{proposition}
\label{regularity.near.infinity}
The singular set $\mathcal{S}$ is compact. The second fundamental form of $\Sigma$ satisfies $|h_\Sigma| \leq C \, r^{-1}$ near infinity. For every nonnegative integer $m$, the $m$-th order covariant derivative of the second fundamental form of $\Sigma$ is bounded by $C(m) \, r^{-m-1}$ near infinity. 
\end{proposition}

\textbf{Proof.} 
This follows from Lemma \ref{density.is.close.to.1} together with Allard's regularity theorem (see \cite{Allard} or \cite{Simon}) and standard interior estimates. \\

\begin{corollary}
\label{graphical}
Near infinity, the hypersurface $\Sigma$ can be written as a graph $x_n = u(x_1,\hdots,x_{n-1})$. The function $u$ is bounded. For every nonnegative integer $m$, the $m$-th order derivatives of $u$ are bounded by $C(m) \, (x_1^2+\hdots+x_{n-1}^2)^{-\frac{m}{2}}$ near infinity.
\end{corollary}

\textbf{Proof.} 
This follows by combining Proposition \ref{regularity.near.infinity}, the density bound in Lemma \ref{density.is.close.to.1}, and the fact that $\Sigma$ is contained in a slab. \\

\begin{proposition}
\label{asymptotic.expansion.of.u}
There exist real numbers $c_0$ and $\mu_0$ such that 
\begin{align*} 
&\Big | u(x_1,\hdots,x_{n-1}) - \big ( c_0+\mu_0 \, (x_1^2+\hdots+x_{n-1}^2)^{\frac{3-n}{2}} \big ) \Big | \\ 
&\leq O \big ( (x_1^2+\hdots+x_{n-1}^2)^{\frac{3-n-\hat{\delta}}{2}} \big ), 
\end{align*} 
and we have analogous estimates for the higher derivatives of $u$.
\end{proposition}

\textbf{Proof.} 
The hypersurface $\Sigma$ satisfies $H_\Sigma + \langle \nabla \log \hat{\rho},\nu_\Sigma \rangle = 0$ near infinity. Hence, the restriction of the coordinate function $x_n$ to $\Sigma$ satisfies 
\begin{equation} 
\label{pde.for.restriction.of.x_n.to.Sigma}
\Delta_\Sigma x_n + \langle \nabla^\Sigma \log \hat{\rho},\nabla^\Sigma x_n \rangle = \Delta x_n - (D^2 x_n)(\nu_\Sigma,\nu_\Sigma) + \langle \nabla \log \hat{\rho},\nabla x_n \rangle 
\end{equation}
near infinity. 

Since $\Sigma$ is contained in a slab, we know that $|x_n| \leq C$ at each point in $\Sigma$. From this, we deduce that 
\begin{equation} 
\label{ambient.Laplacian}
|\Delta x_n| \leq O(r^{1-n-\hat{\delta}}), 
\end{equation}
\begin{equation}
\label{Hessian.in.direction.of.nabla.x_n}
|(D^2 x_n)(\nabla x_n,\nabla x_n)| \leq O(r^{1-n-\hat{\delta}}), 
\end{equation}
and 
\begin{equation} 
\label{drift.term}
|\langle \nabla \log \hat{\rho},\nabla x_n \rangle| \leq O(r^{1-n-\hat{\delta}}) 
\end{equation}
along $\Sigma$. The inequality (\ref{Hessian.in.direction.of.nabla.x_n}) implies 
\begin{equation}
\label{Hessian.in.normal.direction}
|(D^2 x_n)(\nu_\Sigma,\nu_\Sigma)| \leq O(r^{1-n-\hat{\delta}}) 
\end{equation}
along $\Sigma$. Combining (\ref{pde.for.restriction.of.x_n.to.Sigma}), (\ref{ambient.Laplacian}), (\ref{drift.term}), and (\ref{Hessian.in.normal.direction}), we conclude that 
\[|\Delta_\Sigma x_n + \langle \nabla^\Sigma \log \hat{\rho},\nabla^\Sigma x_n \rangle| \leq O(r^{1-n-\hat{\delta}})\] 
on $\Sigma$. Since $x_n = u(x_1,\hdots,x_{n-1})$ at along $\Sigma$, it follows that 
\[|\Delta_\Sigma u(x_1,\hdots,x_{n-1}) + \langle \nabla^\Sigma \log \hat{\rho},\nabla^\Sigma u(x_1,\hdots,x_{n-1}) \rangle| \leq O(r^{1-n-\hat{\delta}})\] 
on $\Sigma$. Therefore, the function $u$ satisfies a linear PDE of the form 
\begin{align*} 
&\sum_{i,j=1}^{n-1} a_{ij}(x_1,\hdots,x_{n-1}) \, \partial_i \partial_j u + \sum_{j=1}^n b_j(x_1,\hdots,x_{n-1}) \, \partial_j u(x_1,\hdots,x_{n-1}) \\ 
&= G(x_1,\hdots,x_{n-1}) 
\end{align*}
near infinity, where
\[\max_{1 \leq i,j \leq n-1} |a_{ij}(x_1,\hdots,x_{n-1}) - \delta_{ij}| \leq O \big ( (x_1^2+\hdots+x_{n-1}^2)^{-\frac{1}{2}} \big ),\] 
\[\max_{1 \leq j \leq n-1} |b_j(x_1,\hdots,x_{n-1})| \leq O \big ( (x_1^2+\hdots+x_{n-1}^2)^{-1} \big ),\] 
and 
\[|G(x_1,\hdots,x_{n-1})| \leq O \big ( (x_1^2+\hdots+x_{n-1}^2)^{\frac{1-n-\hat{\delta}}{2}} \big ).\] 
Moreover, we have analogous estimates for the higher derivatives of $a_{ij}$, $b_j$, and $G$. Since $u$ is bounded, the assertion follows from standard results about linear PDE (see \cite{Meyers}). This completes the proof of Proposition \ref{asymptotic.expansion.of.u}. \\

\subsection{The stability inequality for the $\mu$-bubble} In this subsection, we show that $\hat{\Omega}$ satisfies a stability inequality. In the first step, we state the stability inequality for $\hat{\Omega}^{(j)}$. In the second step, we will pass to the limit as $j \to \infty$.

\begin{proposition}
\label{stability.inequality.for.hat.Omega_j}
Let $j$ be a large integer. Suppose that $a$ is a real number and $V$ is a smooth vector field on $E$ such that $V = a \, \frac{\partial}{\partial x_n}$ in a neighborhood of $W_{s_j} \cap E_{\text{\rm slab},3\lambda_0}$. Then 
\begin{align*} 
&\frac{1}{2} \int_{\partial^* \hat{\Omega}^{(j)} \cap E_{\text{\rm slab},4\lambda_0}^{(j)}} \hat{\rho} \, \sum_{k=1}^{n-1} (\mathscr{L}_V \mathscr{L}_V g)(e_k,e_k) \, d\mathcal{H}^{n-1} \\ 
&+ \int_{\partial^* \hat{\Omega}^{(j)} \cap E_{\text{\rm slab},4\lambda_0}^{(j)}} V(V(\hat{\rho})) \, d\mathcal{H}^{n-1} \\ 
&- \frac{1}{2} \int_{\partial^* \hat{\Omega}^{(j)} \cap E_{\text{\rm slab},4\lambda_0}^{(j)}} \hat{\rho} \, \sum_{k,l=1}^{n-1} (\mathscr{L}_V g)(e_k,e_l) \, (\mathscr{L}_V g)(e_k,e_l) \, d\mathcal{H}^{n-1} \\ 
&+ \frac{1}{4} \int_{\partial^* \hat{\Omega}^{(j)} \cap E_{\text{\rm slab},4\lambda_0}^{(j)}} \hat{\rho} \, \sum_{k,l=1}^{n-1} (\mathscr{L}_V g)(e_k,e_k) \, (\mathscr{L}_V g)(e_l,e_l) \, d\mathcal{H}^{n-1} \\ 
&+ \int_{\partial^* \hat{\Omega}^{(j)} \cap E_{\text{\rm slab},4\lambda_0}^{(j)}} V(\hat{\rho}) \, \sum_{k=1}^{n-1} (\mathscr{L}_V g)(e_k,e_k) \, d\mathcal{H}^{n-1} \\ 
&\geq \int_{\hat{\Omega}^{(j)}} \text{\rm div}(\text{\rm div}(\hat{\rho} \, \Phi \, V) \, V) \, d\mathcal{H}^n. 
\end{align*}
\end{proposition}

\textbf{Proof.}
This follows from the fact that $\hat{\Omega}^{(j)}$ is a minimizer of the functional $\mathcal{F}^{(j)}$. \\

\begin{proposition}
\label{stability.inequality.for.hat.Omega}
Suppose that $a$ is a real number and $V$ is a smooth vector field on $E$ such that $V = a \, \frac{\partial}{\partial x_n}$ near infinity. Then 
\begin{align*} 
&\frac{1}{2} \int_{\partial^* \hat{\Omega} \cap E_{\text{\rm slab},4\lambda_0}} \hat{\rho} \, \sum_{k=1}^{n-1} (\mathscr{L}_V \mathscr{L}_V g)(e_k,e_k) \, d\mathcal{H}^{n-1} \\ 
&+ \int_{\partial^* \hat{\Omega} \cap E_{\text{\rm slab},4\lambda_0}} V(V(\hat{\rho})) \, d\mathcal{H}^{n-1} \\ 
&- \frac{1}{2} \int_{\partial^* \hat{\Omega} \cap E_{\text{\rm slab},4\lambda_0}} \hat{\rho} \, \sum_{k,l=1}^{n-1} (\mathscr{L}_V g)(e_k,e_l) \, (\mathscr{L}_V g)(e_k,e_l) \, d\mathcal{H}^{n-1} \\ 
&+ \frac{1}{4} \int_{\partial^* \hat{\Omega} \cap E_{\text{\rm slab},4\lambda_0}} \hat{\rho} \, \sum_{k,l=1}^{n-1} (\mathscr{L}_V g)(e_k,e_k) \, (\mathscr{L}_V g)(e_l,e_l) \, d\mathcal{H}^{n-1} \\ 
&+ \int_{\partial^* \hat{\Omega} \cap E_{\text{\rm slab},4\lambda_0}} V(\hat{\rho}) \, \sum_{k=1}^{n-1} (\mathscr{L}_V g)(e_k,e_k) \, d\mathcal{H}^{n-1} \\ 
&\geq \int_{\hat{\Omega}} \text{\rm div}(\text{\rm div}(\hat{\rho} \, \Phi \, V) \, V) \, d\mathcal{H}^n. 
\end{align*}
\end{proposition}

\textbf{Proof.}
Note that $|\mathscr{L}_V g| \leq C \, |a| \, r^{1-n}$, $|\mathscr{L}_V \mathscr{L}_V g| \leq C \, |a|^2 \, r^{-n}$, $|V(\hat{\rho})| \leq C \, |a| \, r^{1-n}$, $|V(V(\hat{\rho}))| \leq C \, |a|^2 \, r^{-n}$ near infinity.

We now consider a large number $\bar{s}$. In the following, we assume that $j$ is chosen sufficiently large depending on $\bar{s}$. Using Lemma \ref{area.growth.of.hat.Omega_j}, we can bound 
\begin{align*} 
&\frac{1}{2} \int_{\partial^* \hat{\Omega}^{(j)} \cap E_{\text{\rm slab},4\lambda_0}^{(j)} \cap \{x_1^2+\hdots+x_{n-1}^2 > \bar{s}^2\}} \hat{\rho} \, \sum_{k=1}^{n-1} (\mathscr{L}_V \mathscr{L}_V g)(e_k,e_k) \, d\mathcal{H}^{n-1} \\ 
&+ \int_{\partial^* \hat{\Omega}^{(j)} \cap E_{\text{\rm slab},4\lambda_0}^{(j)} \cap \{x_1^2+\hdots+x_{n-1}^2 > \bar{s}^2\}} V(V(\hat{\rho})) \, d\mathcal{H}^{n-1} \\ 
&- \frac{1}{2} \int_{\partial^* \hat{\Omega}^{(j)} \cap E_{\text{\rm slab},4\lambda_0}^{(j)} \cap \{x_1^2+\hdots+x_{n-1}^2 > \bar{s}^2\}} \hat{\rho} \, \sum_{k,l=1}^{n-1} (\mathscr{L}_V g)(e_k,e_l) \, (\mathscr{L}_V g)(e_k,e_l) \, d\mathcal{H}^{n-1} \\ 
&+ \frac{1}{4} \int_{\partial^* \hat{\Omega}^{(j)} \cap E_{\text{\rm slab},4\lambda_0}^{(j)} \cap \{x_1^2+\hdots+x_{n-1}^2 > \bar{s}^2\}} \hat{\rho} \, \sum_{k,l=1}^{n-1} (\mathscr{L}_V g)(e_k,e_k) \, (\mathscr{L}_V g)(e_l,e_l) \, d\mathcal{H}^{n-1} \\ 
&+ \int_{\partial^* \hat{\Omega}^{(j)} \cap E_{\text{\rm slab},4\lambda_0}^{(j)} \cap \{x_1^2+\hdots+x_{n-1}^2 > \bar{s}^2\}} V(\hat{\rho}) \, \sum_{k=1}^{n-1} (\mathscr{L}_V g)(e_k,e_k) \, d\mathcal{H}^{n-1} \\ 
&\leq C \, |a|^2 \, \bar{s}^{-1}, 
\end{align*} 
where $C$ is independent of $\bar{s}$. Combining this inequality with Proposition \ref{stability.inequality.for.hat.Omega_j}, we obtain 
\begin{align*} 
&\frac{1}{2} \int_{\partial^* \hat{\Omega}^{(j)} \cap E_{\text{\rm slab},4\lambda_0}^{(j)} \setminus \{x_1^2+\hdots+x_{n-1}^2 > \bar{s}^2\}} \hat{\rho} \, \sum_{k=1}^{n-1} (\mathscr{L}_V \mathscr{L}_V g)(e_k,e_k) \, d\mathcal{H}^{n-1} \\ 
&+ \int_{\partial^* \hat{\Omega}^{(j)} \cap E_{\text{\rm slab},4\lambda_0}^{(j)} \setminus \{x_1^2+\hdots+x_{n-1}^2 > \bar{s}^2\}} V(V(\hat{\rho})) \, d\mathcal{H}^{n-1} \\ 
&- \frac{1}{2} \int_{\partial^* \hat{\Omega}^{(j)} \cap E_{\text{\rm slab},4\lambda_0}^{(j)} \setminus \{x_1^2+\hdots+x_{n-1}^2 > \bar{s}^2\}} \hat{\rho} \, \sum_{k,l=1}^{n-1} (\mathscr{L}_V g)(e_k,e_l) \, (\mathscr{L}_V g)(e_k,e_l) \, d\mathcal{H}^{n-1} \\ 
&+ \frac{1}{4} \int_{\partial^* \hat{\Omega}^{(j)} \cap E_{\text{\rm slab},4\lambda_0}^{(j)} \setminus \{x_1^2+\hdots+x_{n-1}^2 > \bar{s}^2\}} \hat{\rho} \, \sum_{k,l=1}^{n-1} (\mathscr{L}_V g)(e_k,e_k) \, (\mathscr{L}_V g)(e_l,e_l) \, d\mathcal{H}^{n-1} \\ 
&+ \int_{\partial^* \hat{\Omega}^{(j)} \cap E_{\text{\rm slab},4\lambda_0}^{(j)} \setminus \{x_1^2+\hdots+x_{n-1}^2 > \bar{s}^2\}} V(\hat{\rho}) \, \sum_{k=1}^{n-1} (\mathscr{L}_V g)(e_k,e_k) \, d\mathcal{H}^{n-1} \\ 
&\geq \int_{\hat{\Omega}^{(j)}} \text{\rm div}(\text{\rm div}(\hat{\rho} \, \Phi \, V) \, V) \, d\mathcal{H}^n - C \, |a|^2 \, \bar{s}^{-1}, 
\end{align*} 
where $C$ is independent of $\bar{s}$. In the next step, we send $j \to \infty$, keeping $\bar{s}$ fixed. It follows from Theorem 21.14 in \cite{Maggi} that the $(n-1)$-dimensional Hausdorff measure on $\partial^* \hat{\Omega}^{(j)} \cap E_{\text{\rm slab},4\lambda_0}$ converges (in the sense of weak convergence of measures) to the $(n-1)$-dimensional Hausdorff measure on $\partial^* \hat{\Omega} \cap E_{\text{\rm slab},4\lambda_0}$. In other words, there is no mass drop. Reshetnyak's continuity theorem (see e.g. \cite{Spector}) now implies that the varifold associated with $\hat{\Omega}^{(j)}$ converges weakly to the varifold associated with $\hat{\Omega}$. This implies 
\begin{align*} 
&\frac{1}{2} \int_{\partial^* \hat{\Omega} \cap E_{\text{\rm slab},4\lambda_0} \setminus \{x_1^2+\hdots+x_{n-1}^2 > \bar{s}^2\}} \hat{\rho} \, \sum_{k=1}^{n-1} (\mathscr{L}_V \mathscr{L}_V g)(e_k,e_k) \, d\mathcal{H}^{n-1} \\ 
&+ \int_{\partial^* \hat{\Omega} \cap E_{\text{\rm slab},4\lambda_0} \setminus \{x_1^2+\hdots+x_{n-1}^2 > \bar{s}^2\}} V(V(\hat{\rho})) \, d\mathcal{H}^{n-1} \\ 
&- \frac{1}{2} \int_{\partial^* \hat{\Omega} \cap E_{\text{\rm slab},4\lambda_0} \setminus \{x_1^2+\hdots+x_{n-1}^2 > \bar{s}^2\}} \hat{\rho} \, \sum_{k,l=1}^{n-1} (\mathscr{L}_V g)(e_k,e_l) \, (\mathscr{L}_V g)(e_k,e_l) \, d\mathcal{H}^{n-1} \\ 
&+ \frac{1}{4} \int_{\partial^* \hat{\Omega} \cap E_{\text{\rm slab},4\lambda_0} \setminus \{x_1^2+\hdots+x_{n-1}^2 > \bar{s}^2\}} \hat{\rho} \, \sum_{k,l=1}^{n-1} (\mathscr{L}_V g)(e_k,e_k) \, (\mathscr{L}_V g)(e_l,e_l) \, d\mathcal{H}^{n-1} \\ 
&+ \int_{\partial^* \hat{\Omega} \cap E_{\text{\rm slab},4\lambda_0} \setminus \{x_1^2+\hdots+x_{n-1}^2 > \bar{s}^2\}} V(\hat{\rho}) \, \sum_{k=1}^{n-1} (\mathscr{L}_V g)(e_k,e_k) \, d\mathcal{H}^{n-1} \\ 
&\geq \int_{\hat{\Omega}} \text{\rm div}(\text{\rm div}(\hat{\rho} \, \Phi \, V) \, V) \, d\mathcal{H}^n - C \, |a|^2 \, \bar{s}^{-1}, 
\end{align*} 
where $C$ is independent of $\bar{s}$. The assertion follows now by sending $\bar{s} \to \infty$. This completes the proof of Proposition \ref{stability.inequality.for.hat.Omega}. \\

\begin{corollary}
\label{stability.inequality.for.normal.variations}
Suppose that $a$ is a real number and $f$ is a smooth test function on $\Sigma$ with the property that $f$ vanishes near the singular set and $f = a \, \langle \frac{\partial}{\partial x_n},\nu_\Sigma \rangle$ near infinity. Then 
\begin{align*} 
&\int_\Sigma \hat{\rho} \, |\nabla^\Sigma f|^2 - \int_\Sigma \hat{\rho} \, \text{\rm Ric}(\nu_\Sigma,\nu_\Sigma) \, f^2 - \int_\Sigma \hat{\rho} \, |h_\Sigma|^2 \, f^2 \\ 
&+ \int_\Sigma \hat{\rho} \, (D^2 \log \hat{\rho})(\nu_\Sigma,\nu_\Sigma) \, f^2 - \int_\Sigma \hat{\rho} \, \langle \nabla \Phi,\nu_\Sigma \rangle \, f^2 \geq 0. 
\end{align*} 
\end{corollary}

\textbf{Proof.} 
We can find a smooth vector field $V$ on $E$ such that $\langle V,\nu_\Sigma \rangle = f$ at each point on $\Sigma$, $V=0$ in a neighborhood of the singular set, and $V = a \, \frac{\partial}{\partial x_n}$ near infinity. We define a vector field $W$ on $E$ by $W = D_V V$. Since $V = a \, \frac{\partial}{\partial x_n}$ near infinity, it follows that $W = a^2 \, D_{\frac{\partial}{\partial x_n}} \frac{\partial}{\partial x_n}$ near infinity. This implies $|W| \leq C \, |a|^2 \, r^{1-n}$ near infinity. 

In the next step, we define a tangential vector field $Z$ along $\Sigma$ by 
\[Z = D_{V^{\text{\rm tan}}}^\Sigma (V^{\text{\rm tan}}) - \text{\rm div}_\Sigma(V^{\text{\rm tan}}) \, V^{\text{\rm tan}} + 2 \sum_{k=1}^{n-1} h_\Sigma(V^{\text{\rm tan}},e_k) \, \langle V,\nu_\Sigma \rangle \, e_k.\] 
Using Proposition \ref{asymptotic.expansion.of.u}, we obtain $|h_\Sigma| \leq C \, r^{1-n}$, $|V^{\text{\rm tan}}| \leq C \, |a| \, r^{2-n}$, and $|Z| \leq C \, |a|^2 \, r^{3-2n}$ near infinity. Proposition \ref{divergence.formula.for.second.variation} gives 
\begin{align*} 
&\hat{\rho} \, |\nabla^\Sigma f|^2 - \hat{\rho} \, (\text{\rm Ric}(\nu_\Sigma,\nu_\Sigma) + |h_\Sigma|^2) \, f^2 + (D^2 \hat{\rho})(\nu_\Sigma,\nu_\Sigma) \, f^2 - \hat{\rho}^{-1} \, \langle \nabla \hat{\rho},\nu_\Sigma \rangle^2 \, f^2 \\ 
&+ \text{\rm div}_\Sigma(\hat{\rho} \, W^{\text{\rm tan}}) - \text{\rm div}_\Sigma(\hat{\rho} \, Z) + \text{\rm div}_\Sigma(\langle V^{\text{\rm tan}},\nabla^\Sigma \hat{\rho} \rangle \, V^{\text{\rm tan}}) \\ 
&- \hat{\rho} \, \langle \nabla \Phi,\nu_\Sigma \rangle \, f^2 + \text{\rm div}(\hat{\rho} \, \Phi \, V) \, \langle V,\nu_\Sigma \rangle + \text{\rm div}_\Sigma(\hat{\rho} \, \Phi \, \langle V,\nu_\Sigma \rangle \, V^{\text{\rm tan}}) \\ 
&= \frac{1}{2} \, \hat{\rho} \sum_{k=1}^{n-1} (\mathscr{L}_V \mathscr{L}_V g)(e_k,e_k) + V(V(\hat{\rho})) \\ 
&- \frac{1}{2} \, \hat{\rho} \sum_{k,l=1}^{n-1} (\mathscr{L}_V g)(e_k,e_l) \, (\mathscr{L}_V g)(e_k,e_l) \\ 
&+ \frac{1}{4} \, \hat{\rho} \sum_{k,l=1}^{n-1} (\mathscr{L}_V g)(e_k,e_k) \, (\mathscr{L}_V g)(e_l,e_l) \\ 
&+ V(\hat{\rho}) \sum_{k=1}^{n-1} (\mathscr{L}_V g)(e_k,e_k) 
\end{align*} 
at each point on $\Sigma$. Integrating this identity over $\Sigma$ gives 
\begin{align*} 
&\int_\Sigma \hat{\rho} \, |\nabla^\Sigma f|^2 - \int_\Sigma \hat{\rho} \, (\text{\rm Ric}(\nu_\Sigma,\nu_\Sigma) + |h_\Sigma|^2) \, f^2 \\ 
&+ \int_\Sigma (D^2 \hat{\rho})(\nu_\Sigma,\nu_\Sigma) \, f^2 - \int_\Sigma \hat{\rho}^{-1} \, \langle \nabla \hat{\rho},\nu_\Sigma \rangle^2 \, f^2 \\ 
&- \int_\Sigma \hat{\rho} \, \langle \nabla \Phi,\nu_\Sigma \rangle \, f^2 + \int_\Sigma \text{\rm div}(\hat{\rho} \, \Phi \, V) \, \langle V,\nu_\Sigma \rangle \\ 
&= \frac{1}{2} \int_\Sigma \hat{\rho} \sum_{k=1}^{n-1} (\mathscr{L}_V \mathscr{L}_V g)(e_k,e_k) + \int_\Sigma V(V(\hat{\rho})) \\ 
&- \frac{1}{2} \int_\Sigma \hat{\rho} \sum_{k,l=1}^{n-1} (\mathscr{L}_V g)(e_k,e_l) \, (\mathscr{L}_V g)(e_k,e_l) \\ 
&+ \frac{1}{4} \int_\Sigma \hat{\rho} \sum_{k,l=1}^{n-1} (\mathscr{L}_V g)(e_k,e_k) \, (\mathscr{L}_V g)(e_l,e_l) \\ 
&+ \int_\Sigma V(\hat{\rho}) \sum_{k=1}^{n-1} (\mathscr{L}_V g)(e_k,e_k). 
\end{align*} 
The assertion follows now from Proposition \ref{stability.inequality.for.hat.Omega}. This completes the proof of Corollary \ref{stability.inequality.for.normal.variations}. \\

\begin{remark}
From now on, we will no longer work with the domain $\hat{\Omega}$, but we will work exclusively with the hypersurface $\Sigma$. The hypersurface $\Sigma$ may be disconnected. If $\Sigma$ is disconnected, then we focus on the connected component of $\Sigma$ that contains the asymptotically planar end. We discard all the other connected components of $\Sigma$.
\end{remark}

\subsection{A generalized Schoen-Yau identity} In this subsection, we prove a generalized Schoen-Yau identity on $\Sigma$.

\begin{proposition}
\label{generalized.Schoen.Yau.identity}
We have 
\begin{align*} 
&\text{\rm Ric}(\nu_\Sigma,\nu_\Sigma) +  |h_\Sigma|^2 - (D^2 \log \hat{\rho})(\nu_\Sigma,\nu_\Sigma) + \langle \nabla \Phi,\nu_\Sigma \rangle \\ 
&+ \frac{1}{2} \, \Big ( R_\Sigma - 2 \, \Delta_\Sigma \log \hat{\rho} - \frac{n}{n+1} \, |\nabla^\Sigma \log \hat{\rho}|^2 \Big ) \geq \hat{Q}
\end{align*} 
at each point on $\Sigma$. Here, $\hat{Q}$ is defined as in Lemma \ref{Phi}.
\end{proposition}

\textbf{Proof.}
The hypersurface $\Sigma$ satisfies 
\begin{equation} 
\label{first.variation.for.weighted.mu.bubble}
H_\Sigma + \langle \nabla \log \hat{\rho},\nu_\Sigma \rangle = \Phi. 
\end{equation}
The Gauss equations imply that 
\begin{equation} 
\label{Gauss.equations}
\frac{1}{2} \, R - \text{\rm Ric}(\nu_\Sigma,\nu_\Sigma) - \frac{1}{2} \, R_\Sigma + \frac{1}{2} \, H_\Sigma^2 - \frac{1}{2} \, |h_\Sigma|^2 = 0 
\end{equation}
at each point on $\Sigma$. Moreover, 
\begin{equation} 
\label{Laplacian.of.log.hat.rho} 
\Delta_\Sigma \log \hat{\rho} = \Delta \log \hat{\rho} - (D^2 \log \hat{\rho})(\nu_\Sigma,\nu_\Sigma) - H_\Sigma \, \langle \nabla \log \hat{\rho},\nu_\Sigma \rangle 
\end{equation} 
and 
\begin{equation} 
\label{gradient.of.log.hat.rho} 
|\nabla^\Sigma \log \hat{\rho}|^2 = |\nabla \log \hat{\rho}|^2 - \langle \nabla \log \hat{\rho},\nu_\Sigma \rangle^2. 
\end{equation} 
This gives 
\begin{align*} 
&\text{\rm Ric}(\nu_\Sigma,\nu_\Sigma) +  |h_\Sigma|^2 - (D^2 \log \hat{\rho})(\nu_\Sigma,\nu_\Sigma) + \langle \nabla \Phi,\nu_\Sigma \rangle \\ 
&+ \frac{1}{2} \, \Big ( R_\Sigma - 2 \, \Delta_\Sigma \log \hat{\rho} - \frac{n}{n+1} \, |\nabla^\Sigma \log \hat{\rho}|^2 \Big ) \\ 
&= \frac{1}{2} \, R + \frac{1}{2} \, H_\Sigma^2 + \frac{1}{2} \, |h_\Sigma|^2 + \langle \nabla \Phi,\nu_\Sigma \rangle \\ 
&- \Delta_\Sigma \log \hat{\rho} - (D^2 \log \hat{\rho})(\nu_\Sigma,\nu_\Sigma) - \frac{n}{2(n+1)} \, |\nabla^\Sigma \log \hat{\rho}|^2 \\ 
&= \frac{1}{2} \, R + \frac{1}{2} \, H_\Sigma^2 + \frac{1}{2} \, |h_\Sigma|^2 + \langle \nabla \Phi,\nu_\Sigma \rangle \\ 
&- \Delta \log \hat{\rho} + H_\Sigma \, \langle \nabla \log \hat{\rho},\nu_\Sigma \rangle \\ 
&- \frac{n}{2(n+1)} \, |\nabla \log \hat{\rho}|^2 + \frac{n}{2(n+1)} \, \langle \nabla \log \hat{\rho},\nu_\Sigma \rangle^2 \\ 
&= \frac{1}{2} \, R + \frac{1}{4} \, \Phi^2 + \frac{1}{2} \, \Big ( |h_\Sigma|^2 - \frac{1}{n-1} \, H_\Sigma^2 \Big ) + \langle \nabla \Phi,\nu_\Sigma \rangle \\ 
&- (\Delta \log \rho + \Delta \log \hat{v}) - \frac{n}{2(n+1)} \, |\nabla \log \rho + \nabla \log \hat{v}|^2 \\ 
&+ \frac{n+1}{4(n-1)} \, \Big ( H_\Sigma + \frac{n-1}{n+1} \, \langle \nabla \log \hat{\rho},\nu_\Sigma \rangle \Big )^2 \\ 
&= \frac{1}{2} \, Q + \frac{1}{4} \, \Phi^2 + \frac{1}{2} \, \Big ( |h_\Sigma|^2 - \frac{1}{n-1} \, H_\Sigma^2 \Big ) + \langle \nabla \Phi,\nu_\Sigma \rangle \\ 
&+ \frac{n+2}{2(n+1)} \, \Big | \nabla \log \hat{v} + \frac{1}{n+2} \, \nabla \log \rho \Big |^2 \\ 
&+ \frac{n+1}{4(n-1)} \, \Big ( H_\Sigma + \frac{n-1}{n+1} \, \langle \nabla \log \hat{\rho},\nu_\Sigma \rangle \Big )^2 \\ 
&\geq \frac{1}{2} \, Q + \frac{1}{4} \, \Phi^2 - |\nabla \Phi| 
\end{align*} 
at each point on $\Sigma$. The first equality follows from (\ref{Gauss.equations}). The second equality follows from (\ref{Laplacian.of.log.hat.rho}) and (\ref{gradient.of.log.hat.rho}). The third equality follows from (\ref{first.variation.for.weighted.mu.bubble}) and the identity $\hat{\rho} = \rho \, \hat{v}$. The fourth equality follows from (\ref{pde.for.log.hat.v}). 

Finally, $Q + \frac{1}{2} \, \Phi^2 - 2 \, |\nabla \Phi| \geq 2 \, \hat{Q}$ by Lemma \ref{Phi}. This completes the proof of Proposition \ref{generalized.Schoen.Yau.identity}. \\

\begin{corollary}
\label{quadratic.form.with.respect.to.hat.g}
Let $f$ be a smooth test function on $\Sigma$ with the property that $f$ vanishes near the singular set and $f$ is constant near infinity. Then 
\begin{align*} 
&\int_\Sigma \hat{\rho} \, |\nabla^\Sigma f|^2 + \frac{1}{2} \int_\Sigma \hat{\rho} \, \Big ( R_\Sigma - 2 \, \Delta_\Sigma \log \hat{\rho} - \frac{n}{n+1} \, |\nabla^\Sigma \log \hat{\rho}|^2 \Big ) \, f^2 \\ 
&\geq \int_\Sigma \hat{\rho} \, \hat{Q} \, f^2. 
\end{align*} 
\end{corollary}

\textbf{Proof.}
Using Proposition \ref{stability.inequality.for.normal.variations} and a standard approximation argument, we obtain 
\begin{align*} 
&\int_\Sigma \hat{\rho} \, |\nabla^\Sigma f|^2 - \int_\Sigma \hat{\rho} \, \text{\rm Ric}(\nu_\Sigma,\nu_\Sigma) \, f^2 - \int_\Sigma \hat{\rho} \, |h_\Sigma|^2 \, f^2 \\ 
&+ \int_\Sigma \hat{\rho} \, (D^2 \log \hat{\rho})(\nu_\Sigma,\nu_\Sigma) \, f^2 - \int_\Sigma \hat{\rho} \, \langle \nabla \Phi,\nu_\Sigma \rangle \, f^2 \geq 0. 
\end{align*} 
The assertion follows now from Proposition \ref{generalized.Schoen.Yau.identity}. This completes the proof of Corollary \ref{quadratic.form.with.respect.to.hat.g}. \\

\begin{corollary}
\label{singular.set.empty}
If $\mathcal{S} = \emptyset$, then $(\Sigma,\hat{g},\hat{\rho},\hat{Q})$ is an $(n-1)$-dataset and its mass (in the sense of Definition \ref{mass.of.dataset}) is equal to $0$.
\end{corollary}

\textbf{Proof.} 
Since $\mathcal{S} = \emptyset$, it follows that $(\Sigma,\hat{g})$ is a complete Riemannian manifold. Using Corollary \ref{quadratic.form.with.respect.to.hat.g}, we can show that $(\Sigma,\hat{g},\hat{\rho},\hat{Q})$ is an $(n-1)$-dataset. We next observe that 
\[\hat{g} = dx_1 \otimes dx_1 + \hdots + dx_{n-1} \otimes dx_{n-1} + O \big ( (x_1^2+\hdots+x_{n-1}^2)^{\frac{2-n}{2}} \big )\] 
and 
\[\hat{\rho} = 1 + O \big ( (x_1^2+\hdots+x_{n-1}^2)^{\frac{2-n}{2}} \big )\] 
near infinity. This shows that $(\Sigma,\hat{g},\hat{\rho},\hat{Q})$ has zero mass (in the sense of Definition \ref{mass.of.dataset}). This completes the proof of Corollary \ref{singular.set.empty}. \\

\subsection{A function on ambient space that blows up at the singular set at a controlled rate} In this subsection, we construct a function that will allow us to apply a conformal blow-up technique in the spirit of Bi-Hao-He-Shi-Zhu \cite{Bi-Hao-He-Shi-Zhu}. 

Throughout this subsection, we assume that the singular set $\mathcal{S}$ is non-empty. In this case, $n \geq 8$. Note that $\mathcal{S}$ is a compact subset of $E$. In particular, $\mathcal{S}$ has positive distance from the boundary $\partial E$. Let us fix a positive real number $t_*$ so that $\sqrt{t_*} < \frac{1}{2} \, d_{(M,g)}(\mathcal{S},\partial E)$ and $\sqrt{t_*} < \frac{1}{2} \, \text{\rm inj}_{(M,g)}(p)$ for each point $p \in \mathcal{S}$. 

Let us fix a nonnegative smooth function $\zeta: (0,\infty) \to \mathbb{R}$ such that 
\[\zeta(t) = \frac{2}{n-3} \, t^{\frac{3-n}{2}} + \frac{4}{2n-7} \, t^{\frac{7-2n}{4}}\] 
for $t \in (0,\frac{t_*}{4}]$ and 
\[\zeta(t) = 0\] 
for $t \in [t_*,\infty)$. In the following, $\Phi$ will denote the function constructed in Lemma \ref{Phi}. 

\begin{lemma}
\label{properties.of.Psi}
We can find a small constant $t_0 \in (0,\frac{t_*}{4})$ with the following significance. Let $p$ be a point in $\mathcal{S}$. Let us define a smooth function $\psi: E \setminus \{p\} \to \mathbb{R}$ by  
\[\psi(x) = \zeta(d_{(M,g)}(p,x)^2)\] 
for all $x \in E \setminus \{p\}$. Then 
\begin{align*} 
&\Delta \psi - (D^2 \psi)(\xi,\xi) + \frac{n-3}{n+1} \, \langle \nabla \log \hat{\rho},\nabla \psi \rangle \\ 
&- \Big ( \Phi - \frac{4}{n+1} \, \langle \nabla \log \hat{\rho},\xi \rangle \Big ) \, \langle \nabla \psi,\xi \rangle \leq -\frac{1}{2} \, d_{(M,g)}(p,x)^{\frac{3-2n}{2}} 
\end{align*} 
for each point $x \in \mathcal{B}_{(M,g)}(p,\sqrt{t_0})$ and every unit vector $\xi \in T_x M$. Moreover, 
\begin{align*} 
&\Delta \psi - (D^2 \psi)(\xi,\xi) + \frac{n-3}{n+1} \, \langle \nabla \log \hat{\rho},\nabla \psi \rangle \\ 
&- \Big ( \Phi - \frac{4}{n+1} \, \langle \nabla \log \hat{\rho},\xi \rangle \Big ) \, \langle \nabla \psi,\xi \rangle \leq C 
\end{align*} 
for each point $x \in E \setminus \{p\}$ and every unit vector $\xi \in T_x M$. 
\end{lemma}

\textbf{Proof.} 
We define a function $\tau: E \to \mathbb{R}$ by $\tau(x) = d_{(M,g)}(p,x)^2$ for all $x \in E$. The function $\tau$ is smooth in $\mathcal{B}_{(M,g)}(p,\sqrt{t_*})$. The gradient and Hessian of $\tau$ satisfy 
\[|d\tau|^2 = 4\tau\] 
and 
\[D^2 \tau \geq 2 \, (1-C_0\tau) \, g\] 
at each point $x \in \mathcal{B}_{(M,g)}(p,\sqrt{t_*})$. This implies 
\begin{align*} 
D^2 \psi 
&= -(\tau^{\frac{1-n}{2}} + \tau^{\frac{3-2n}{4}}) \, D^2 \tau \\ 
&+ \frac{1}{4} \, \big ( 2(n-1) \, \tau^{-\frac{n+1}{2}} + (2n-3) \, \tau^{-\frac{2n+1}{4}} \big ) \, d\tau \otimes d\tau
\end{align*} 
at each point $x \in \mathcal{B}_{(M,g)}(p,\frac{\sqrt{t_*}}{2})$. This gives 
\begin{align*} 
&\Delta \psi - (D^2 \psi)(\xi,\xi) \\ 
&= -(\tau^{\frac{1-n}{2}} + \tau^{\frac{3-2n}{4}}) \, (\Delta \tau - (D^2 \tau)(\xi,\xi)) \\ 
&+ \frac{1}{4} \, \big ( 2(n-1) \, \tau^{-\frac{n+1}{2}} + (2n-3) \, \tau^{-\frac{2n+1}{4}} \big ) \, (|\nabla \tau|^2 - \langle \nabla \tau,\xi \rangle^2) \\ 
&\leq -2(n-1) \, (\tau^{\frac{1-n}{2}} + \tau^{\frac{3-2n}{4}}) \, (1-C_0\tau) \\ 
&+ \big ( 2(n-1) \, \tau^{-\frac{n+1}{2}} + (2n-3) \, \tau^{-\frac{2n+1}{4}} \big ) \, \tau \\ 
&= -\tau^{\frac{3-2n}{4}} + 2(n-1)C_0 \, (\tau^{\frac{3-n}{2}} + \tau^{\frac{7-2n}{4}}) 
\end{align*} 
at each point $x \in \mathcal{B}_{(M,g)}(p,\frac{\sqrt{t_*}}{2})$. Since $\sqrt{t_*} < \frac{1}{2} \, d_{(M,g)}(\mathcal{S},\partial E)$, the function $\Phi$ is uniformly bounded on $\mathcal{B}_{(M,g)}(p,\sqrt{t_*})$. This implies 
\begin{align*} 
&\frac{n-3}{n+1} \, \langle \nabla \log \hat{\rho},\nabla \psi \rangle - \Big ( \Phi - \frac{4}{n+1} \, \langle \nabla \log \hat{\rho},\xi \rangle \Big ) \, \langle \nabla \psi,\xi \rangle \\ 
&\leq C_1 \, |d\psi| \\ 
&= C_1 \, (\tau^{\frac{1-n}{2}} + \tau^{\frac{3-2n}{4}}) \, |d\tau| \\ 
&\leq 2C_1 \, (\tau^{\frac{2-n}{2}} + \tau^{\frac{5-2n}{4}}) 
\end{align*} 
at each point $x \in \mathcal{B}_{(M,g)}(p,\frac{\sqrt{t_*}}{2})$. Putting these facts together, we obtain 
\begin{align*} 
&\Delta \psi - (D^2 \psi)(\xi,\xi) + \frac{n-3}{n+1} \, \langle \nabla \log \hat{\rho},\nabla \psi \rangle \\ 
&- \Big ( \Phi - \frac{4}{n+1} \, \langle \nabla \log \hat{\rho},\xi \rangle \Big ) \, \langle \nabla \psi,\xi \rangle \\ 
&\leq -\tau^{\frac{3-2n}{4}} + 2(n-1)C_0 \, (\tau^{\frac{3-n}{2}} + \tau^{\frac{7-2n}{4}}) \\ 
&+ 2C_1 \, (\tau^{\frac{2-n}{2}} + \tau^{\frac{5-2n}{4}}) 
\end{align*} 
at each point $x \in \mathcal{B}_{(M,g)}(p,\frac{\sqrt{t_*}}{2})$. Hence, if we choose $t_0$ sufficiently small, then the first statement holds. Moreover, the second statement holds in the ball $\mathcal{B}_{(M,g)}(p,\frac{\sqrt{t_*}}{2})$. Finally, it is easy to see that the second statement holds on the set $E \setminus \mathcal{B}_{(M,g)}(p,\frac{\sqrt{t_*}}{2})$. This completes the proof of Lemma \ref{properties.of.Psi}. \\

In the next step, we need an estimate for the Minkowski dimension of the singular set $\mathcal{S}$.

\begin{theorem}[cf. J.~Cheeger, A.~Naber \cite{Cheeger-Naber}, Theorem 5.8] 
\label{Minkowski.dimension}
The singular set $\mathcal{S}$ has Minkowski dimension at most $n-8$. 
\end{theorem}

The bound for the Minkowski dimension of the singular set was originally proved by Cheeger and Naber \cite{Cheeger-Naber} for area-minimizing currents in codimension $1$. Their arguments rely on the monotonicity formula and can be generalized to the setting of $\mu$-bubbles (see \cite{Aiex-McCurdy-Minter},\cite{Focardi-Marchese-Spadaro}).

For each positive integer $l$, we define $t_l = 4^{-l} \, t_0$. For each positive integer $l$, we choose a finite subset $\mathcal{S}^{(l)} \subset \mathcal{S}$ so that the balls $\{\mathcal{B}_{(M,g)}(p,\sqrt{t_l}): p \in \mathcal{S}^{(l)}\}$ are disjoint, and so that $\mathcal{S}^{(l)}$ is maximal with respect to this property. Clearly, 
\[\sum_{p \in \mathcal{S}^{(l)}} \text{\rm vol}_g(\mathcal{B}_{(M,g)}(p,\sqrt{t_l})) \leq \text{\rm vol}_g(\mathcal{N}_{(M,g)}(\mathcal{S},\sqrt{t_l})).\] 
Consequently, the cardinality of $\mathcal{S}^{(l)}$ is bounded from above by 
\[|\mathcal{S}^{(l)}| \leq C \, t_l^{-\frac{n}{2}} \, \text{\rm vol}_g(\mathcal{N}_{(M,g)}(\mathcal{S},\sqrt{t_l})).\] 
On the other hand, since $\mathcal{S}$ has Minkowski dimension at most $n-8$, we know that 
\[\text{\rm vol}_g(\mathcal{N}_{(M,g)}(\mathcal{S},\sqrt{t_l})) \leq C(\mu) \, t_l^{\frac{8-\mu}{2}}\] 
for each $\mu > 0$. Putting these facts together, we obtain 
\[|\mathcal{S}^{(l)}| \leq C(\mu) \, t_l^{\frac{8-n-\mu}{2}}\] 
for each $\mu > 0$. In particular, $\sum_{l=1}^\infty t_l^{\frac{n-5}{2}} \, |\mathcal{S}^{(l)}| < \infty$. 

We now define a function $\Psi: E \setminus \mathcal{S} \to \mathbb{R}$ by 
\begin{equation} 
\label{definition.of.Psi}
\Psi(x) = \sum_{l=1}^\infty \sum_{p \in \mathcal{S}^{(l)}} t_l^{\frac{n-5}{2}} \, \zeta(d_{(M,g)}(p,x)^2) 
\end{equation}
for all points $x \in E \setminus \mathcal{S}$. Since $\sum_{l=1}^\infty t_l^{\frac{n-5}{2}} \, |\mathcal{S}^{(l)}| < \infty$, the series converges and the function $\Psi$ is well-defined. Moreover, using the fact that $\sum_{l=1}^\infty t_l^{\frac{n-5}{2}} \, |\mathcal{S}^{(l)}| < \infty$, it is easy to see that $\Psi$ is a smooth function on $E \setminus \mathcal{S}$. Indeed, for every nonnegative integer $m$, we can bound 
\begin{align*} 
|D^m \Psi| 
&\leq C(m) \sum_{l=1}^\infty \sum_{p \in \mathcal{S}^{(l)}} t_l^{\frac{n-5}{2}} \, d_{(M,g)}(p,x)^{3-n-m} \\ 
&\leq C(m) \sum_{l=1}^\infty t_l^{\frac{n-5}{2}} \, |\mathcal{S}^{(l)}| \, d_{(M,g)}(x,\mathcal{S})^{3-n-m}, 
\end{align*} 
where $D^m$ denotes the covariant derivative of order $m$ with respect to the metric $g$.

\begin{proposition}
\label{lower.bound.for.Psi}
Suppose that $x$ is a point in $E \setminus \mathcal{S}$ satisfying $d_{(M,g)}(x,\mathcal{S}) < \frac{1}{4} \, \sqrt{t_0}$. Then $\Psi(x) \geq \frac{1}{2(n-3)} \, 3^{3-n} \, d_{(M,g)}(x,\mathcal{S})^{-2}$. 
\end{proposition}

\textbf{Proof.}
Since $d_{(M,g)}(x,\mathcal{S}) < \frac{1}{4} \, \sqrt{t_0}$, we can find an integer $l_0 \geq 2$ such that $\sqrt{t_{l_0+1}} \leq d_{(M,g)}(x,\mathcal{S}) < \sqrt{t_{l_0}}$. We can find a point $q \in \mathcal{S}$ such that $d_{(M,g)}(x,q) \leq \sqrt{t_{l_0}}$. In view of the maximality of $\mathcal{S}^{(l_0)}$, we can find a point $p_0 \in \mathcal{S}^{(l_0)}$ such that $\mathcal{B}_{(M,g)}(p_0,\sqrt{t_{l_0}}) \cap \mathcal{B}_{(M,g)}(q,\sqrt{t_{l_0}}) \neq \emptyset$. Using the triangle inequality, we obtain $d_{(M,g)}(p_0,q) \leq 2\sqrt{t_{l_0}}$. This implies 
\[d_{(M,g)}(p_0,x) \leq d_{(M,g)}(p_0,q) + d_{(M,g)}(x,q) \leq 3\sqrt{t_{l_0}}.\] 
Since $l_0 \geq 2$, it follows that $d_{(M,g)}(p_0,x) \leq \sqrt{t_0} \leq \frac{\sqrt{t_*}}{2}$. Thus, we conclude that 
\[\zeta(d_{(M,g)}(p_0,x)^2) \geq \frac{2}{n-3} \, d_{(M,g)}(p_0,x)^{3-n} \geq \frac{2}{n-3} \, 3^{3-n} \, t_{l_0}^{\frac{3-n}{2}}.\] 
This finally implies 
\begin{align*} 
\Psi(x) 
&\geq t_{l_0}^{\frac{n-5}{2}} \, \zeta(d_{(M,g)}(p_0,x)^2) \\ 
&\geq \frac{2}{n-3} \, 3^{3-n} \, t_{l_0}^{-1} \\ 
&\geq \frac{1}{2(n-3)} \, 3^{3-n} \, d_{(M,g)}(x,\mathcal{S})^{-2}. 
\end{align*}
This completes the proof of Proposition \ref{lower.bound.for.Psi}. \\

\begin{proposition}
\label{key.property.of.Psi}
Suppose that $x$ is a point in $E \setminus \mathcal{S}$ satisfying $d_{(M,g)}(x,\mathcal{S}) < \frac{1}{4} \, \sqrt{t_0}$, and $\xi \in T_x M$ is a unit vector. If $d_{(M,g)}(x,\mathcal{S})$ is sufficiently small, then 
\begin{align*} 
&\Delta \Psi - (D^2 \Psi)(\xi,\xi) + \frac{n-3}{n+1} \, \langle \nabla \log \hat{\rho},\nabla \Psi \rangle \\ 
&- \Big ( \Phi - \frac{4}{n+1} \, \langle \nabla \log \hat{\rho},\xi \rangle \Big ) \, \langle \nabla \Psi,\xi \rangle < 0 
\end{align*} 
at the point $x$.
\end{proposition}

\textbf{Proof.}
Since $d_{(M,g)}(x,\mathcal{S}) < \frac{1}{4} \, \sqrt{t_0}$, we can find an integer $l_0 \geq 2$ such that $\sqrt{t_{l_0+1}} \leq d_{(M,g)}(x,\mathcal{S}) < \sqrt{t_{l_0}}$. As above, we can find a point $p_0 \in \mathcal{S}^{(l_0)}$ such that $d_{(M,g)}(p_0,x) \leq 3\sqrt{t_{l_0}}$. Using Lemma \ref{properties.of.Psi}, we obtain 
\begin{align*} 
&\Delta \Psi - (D^2 \Psi)(\xi,\xi) + \frac{n-3}{n+1} \, \langle \nabla \log \hat{\rho},\nabla \Psi \rangle \\ 
&- \Big ( \Phi - \frac{4}{n+1} \, \langle \nabla \log \hat{\rho},\xi \rangle \Big ) \, \langle \nabla \Psi,\xi \rangle \\ 
&\leq -\frac{1}{2} \, t_{l_0}^{\frac{n-5}{2}} \, d_{(M,g)}(p_0,x)^{\frac{3-2n}{2}} + C \sum_{l=1}^\infty t_l^{\frac{n-5}{2}} \, |\mathcal{S}^{(l)}| \\ 
&\leq -\frac{1}{2} \, 3^{\frac{3-2n}{2}} \, t_{l_0}^{-\frac{7}{4}} + C \sum_{l=1}^\infty t_l^{\frac{n-5}{2}} \, |\mathcal{S}^{(l)}|. 
\end{align*} 
Since $\sum_{l=1}^\infty t_l^{\frac{n-5}{2}} \, |\mathcal{S}^{(l)}| < \infty$, the expression on the right hand side is negative if $l_0$ is sufficiently large. This completes the proof of Proposition \ref{key.property.of.Psi}. \\

\begin{corollary}
\label{choice.of.varepsilon_0}
We can find a small positive real number $\varepsilon_0$ with the property that 
\begin{align*} 
&\Delta \Psi - (D^2 \Psi)(\xi,\xi) + \frac{n-3}{n+1} \, \langle \nabla \log \hat{\rho},\nabla \Psi \rangle \\ 
&- \Big ( \Phi - \frac{4}{n+1} \, \langle \nabla \log \hat{\rho},\xi \rangle \Big ) \, \langle \nabla \Psi,\xi \rangle \leq \frac{n-3}{n+1} \, \varepsilon_0^{-1} \, \hat{Q} 
\end{align*} 
for each point $x \in E \setminus \mathcal{S}$ and every unit vector $\xi \in T_x M$.
\end{corollary}

\textbf{Proof.} 
The function $\Psi$ vanishes near $\partial E$ and near infinity. Moreover, Proposition \ref{key.property.of.Psi} implies that 
\begin{align*} 
&\Delta \Psi - (D^2 \Psi)(\xi,\xi) + \frac{n-3}{n+1} \, \langle \nabla \log \hat{\rho},\nabla \Psi \rangle \\ 
&- \Big ( \Phi - \frac{4}{n+1} \, \langle \nabla \log \hat{\rho},\xi \rangle \Big ) \, \langle \nabla \Psi,\xi \rangle < 0 
\end{align*} 
if $d_{(M,g)}(x,\mathcal{S})$ is sufficiently small and $\xi \in T_x M$ is a unit vector. Since the function $\hat{Q}$ is strictly positive, the assertion follows. This completes the proof of Corollary \ref{choice.of.varepsilon_0}. \\

\begin{corollary}
\label{restriction.of.Psi}
Let $\varepsilon_0$ be chosen as in Corollary \ref{choice.of.varepsilon_0}. Then 
\[\Delta_\Sigma \Psi + \frac{n-3}{n+1} \, \langle \nabla^\Sigma \log \hat{\rho},\nabla^\Sigma \Psi \rangle \leq \frac{n-3}{n+1} \, \varepsilon_0^{-1} \, \hat{Q}\] 
at each point on $\Sigma$.
\end{corollary}

\textbf{Proof.} 
The hypersurface $\Sigma$ satisfies 
\[H_\Sigma + \langle \nabla \log \hat{\rho},\nu_\Sigma \rangle = \Phi.\]
This implies 
\begin{align*} 
&\Delta_\Sigma \Psi + \frac{n-3}{n+1} \, \langle \nabla^\Sigma \log \hat{\rho},\nabla^\Sigma \Psi \rangle \\ 
&= \Delta \Psi - (D^2 \Psi)(\nu_\Sigma,\nu_\Sigma) + \frac{n-3}{n+1} \, \langle \nabla \log \hat{\rho},\nabla \Psi \rangle \\ 
&- \Big ( \Phi - \frac{4}{n+1} \, \langle \nabla \log \hat{\rho},\nu_\Sigma \rangle \Big ) \, \langle \nabla \Psi,\nu_\Sigma \rangle. 
\end{align*}
If we apply Proposition \ref{key.property.of.Psi} with $\xi = \nu_\Sigma$, the assertion follows. This completes the proof of Corollary \ref{restriction.of.Psi}. \\

\subsection{The conformal blow-up procedure and the conclusion of the inductive step} In this subsection, we complete the inductive step and conclude the proof of Theorem \ref{pmt}. Throughout this subsection, we assume that the singular set $\mathcal{S}$ is non-empty. In this case, $n \geq 8$. Let $\Psi$ denote the function defined in (\ref{definition.of.Psi}), and let $\varepsilon_0$ be chosen as in Corollary \ref{choice.of.varepsilon_0}. We define a function $w: \Sigma \to \mathbb{R}$ by 
\begin{equation} 
\label{definition.of.w}
w = 1 + \varepsilon_0 \, \Psi|_\Sigma. 
\end{equation}
Since $\Psi$ is nonnegative, we have 
\begin{equation} 
\label{lower.bound.for.w}
w \geq 1
\end{equation}
at each point on $\Sigma$. 

In the next step, we define a conformal metric $\tilde{g}$ on $\Sigma$ by 
\[\tilde{g} = w^{\frac{n+1}{n-3}} \, \hat{g}.\] 
Moreover, we define a positive smooth function $\tilde{\rho}$ on $\Sigma$ by 
\[\tilde{\rho} = w^{-\frac{n+1}{2}} \, \hat{\rho}.\] 
Finally, we define a function $\tilde{Q}$ on $\Sigma$ by 
\[\tilde{Q} = \frac{1}{2} \, w^{-\frac{n+1}{n-3}} \, \hat{Q}.\] 

\begin{lemma}
\label{inequality.for.tilde.Q}
We have 
\[w^{\frac{n+1}{n-3}} \, \tilde{Q} \leq \hat{Q} - \frac{n+1}{2(n-3)} \, w^{-1} \, \Delta_{\hat{g}} w - \frac{1}{2} \, w^{-1} \, \langle d\log \hat{\rho},dw \rangle_{\hat{g}}.\]
\end{lemma}

\textbf{Proof.} 
Using (\ref{definition.of.w}), (\ref{lower.bound.for.w}), and Corollary \ref{restriction.of.Psi}, we obtain 
\[\Delta_\Sigma w + \frac{n-3}{n+1} \, \langle \nabla^\Sigma \log \hat{\rho},\nabla^\Sigma w \rangle \leq \frac{n-3}{n+1} \, \hat{Q} \leq \frac{n-3}{n+1} \, \hat{Q} \, w\] 
at each point on $\Sigma$. The assertion follows from the preceding inequality together with the definition of $\tilde{Q}$. This completes the proof of Lemma \ref{inequality.for.tilde.Q}. \\

\begin{proposition}
\label{quadratic.form.with.respect.to.tilde.g}
Let $f$ be a smooth test function on $\Sigma$ with the property that $f$ vanishes near the singular set and $f$ is constant near infinity. Then 
\begin{align*} 
&\int_\Sigma \tilde{\rho} \, |df|_{\tilde{g}}^2 \, d\text{\rm vol}_{\tilde{g}} + \frac{1}{2} \int_\Sigma \tilde{\rho} \, \Big ( R_{\tilde{g}} - 2 \, \Delta_{\tilde{g}} \log \tilde{\rho} - \frac{n}{n+1} \, |d \log \tilde{\rho}|_{\tilde{g}}^2 \Big ) \, f^2 \, d\text{\rm vol}_{\tilde{g}} \\ 
&\geq \int_\Sigma \tilde{\rho} \, \tilde{Q} \, f^2 \, d\text{\rm vol}_{\tilde{g}} 
\end{align*}
\end{proposition}

\textbf{Proof.}
Note that 
\[w^{\frac{n+1}{n-3}} \, |df|_{\tilde{g}}^2 = |df|_{\hat{g}}^2.\] 
The standard formula for the change of the scalar curvature under a conformal change of the metric gives 
\begin{align*} 
w^{\frac{n+1}{n-3}} \, R_{\tilde{g}} 
&= R_{\hat{g}} - \frac{4(n-2)}{n-3} \, w^{-\frac{n+1}{4}} \, \Delta_{\hat{g}} (w^{\frac{n+1}{4}}) \\ 
&= R_{\hat{g}} - \frac{(n-2)(n+1)}{n-3} \, w^{-1} \, \Delta_{\hat{g}} w - \frac{(n-2)(n+1)}{4} \, w^{-2} \, |dw|_{\hat{g}}^2. 
\end{align*}
Moreover, 
\begin{align*} 
w^{\frac{n+1}{n-3}} \, \Delta_{\tilde{g}} \log \tilde{\rho} 
&= w^{-\frac{n+1}{2}} \, \text{\rm div}_{\hat{g}}(w^{\frac{n+1}{2}} \, d\log \tilde{\rho}) \\ 
&= \Delta_{\hat{g}} \log \tilde{\rho} + \frac{n+1}{2} \, \langle d\log \tilde{\rho},d\log w \rangle \\ 
&= \Delta_{\hat{g}} \log \hat{\rho} - \frac{n+1}{2} \, \Delta_{\hat{g}} \log w + \frac{n+1}{2} \, \langle d\log \hat{\rho},d\log w \rangle_{\hat{g}} \\ 
&- \frac{(n+1)^2}{4} \, |d\log w|_{\hat{g}}^2 \\ 
&= \Delta_{\hat{g}} \log \hat{\rho} - \frac{n+1}{2} \, w^{-1} \, \Delta_{\hat{g}} w + \frac{n+1}{2} \, w^{-1} \, \langle d\log \hat{\rho},dw \rangle_{\hat{g}} \\ 
&- \frac{(n-1)(n+1)}{4} \, w^{-2} \, |dw|_{\hat{g}}^2 
\end{align*}
and 
\begin{align*} 
w^{\frac{n+1}{n-3}} \, |d\log \tilde{\rho}|_{\tilde{g}}^2 
&= |d \log \tilde{\rho}|_{\hat{g}}^2 \\ 
&= \Big | d \log \hat{\rho} - \frac{n+1}{2} \, d\log w \Big |_{\hat{g}}^2 \\ 
&= |d\log \hat{\rho}|_{\hat{g}}^2 - (n+1) \, w^{-1} \, \langle d\log \hat{\rho},dw \rangle_{\hat{g}} \\ 
&+ \frac{(n+1)^2}{4} \, w^{-2} \, |dw|_{\hat{g}}^2. 
\end{align*}
Using these identities together with Lemma \ref{inequality.for.tilde.Q}, we obtain the pointwise inequality 
\begin{align*} 
&w^{\frac{n+1}{n-3}} \, |df|_{\tilde{g}}^2 + \frac{1}{2} \, w^{\frac{n+1}{n-3}} \, \Big ( R_{\tilde{g}} - 2 \, \Delta_{\tilde{g}} \log \tilde{\rho} - \frac{n}{n+1} \, |d \log \tilde{\rho}|_{\tilde{g}}^2 \Big ) \, f^2 - w^{\frac{n+1}{n-3}} \, \tilde{Q} \, f^2 \\ 
&\geq |df|_{\hat{g}}^2 + \frac{1}{2} \, \Big ( R_{\hat{g}} - 2 \, \Delta_{\hat{g}} \log \hat{\rho} - \frac{n}{n+1} \, |d\log \hat{\rho}|_{\hat{g}}^2 \Big ) \, f^2 - \hat{Q} \, f^2. 
\end{align*}
This finally implies 
\begin{align*} 
&\int_\Sigma \tilde{\rho} \, |df|_{\tilde{g}}^2 \, d\text{\rm vol}_{\tilde{g}} + \frac{1}{2} \int_\Sigma \tilde{\rho} \, \Big ( R_{\tilde{g}} - 2 \, \Delta_{\tilde{g}} \log \tilde{\rho} - \frac{n}{n+1} \, |d \log \tilde{\rho}|_{\tilde{g}}^2 \Big ) \, f^2 \, d\text{\rm vol}_{\tilde{g}} \\ 
&- \int_\Sigma \tilde{\rho} \, \tilde{Q} \, f^2 \, d\text{\rm vol}_{\tilde{g}} \\ 
&\geq \int_\Sigma \hat{\rho} \, |df|_{\hat{g}}^2 \, d\text{\rm vol}_{\hat{g}} + \frac{1}{2} \int_\Sigma \hat{\rho} \, \Big ( R_{\hat{g}} - 2 \, \Delta_{\hat{g}} \log \hat{\rho} - \frac{n}{n+1} \, |d \log \hat{\rho}|_{\hat{g}}^2 \Big ) \, f^2 \, d\text{\rm vol}_{\hat{g}} \\ 
&- \int_\Sigma \hat{\rho} \, \hat{Q} \, f^2 \, d\text{\rm vol}_{\hat{g}}. 
\end{align*}
The expression on the right hand side is nonnegative by Corollary \ref{quadratic.form.with.respect.to.hat.g}. This completes the proof of Proposition \ref{quadratic.form.with.respect.to.tilde.g}. \\

\begin{proposition}
\label{completeness}
The metric $\tilde{g}$ on $\Sigma$ is complete.
\end{proposition}

\textbf{Proof.}
By Proposition \ref{lower.bound.for.Psi}, we can find small positive constants $c_1$ and $c_2$ such that  
\[w(x) \geq c_2 \, d_{(M,g)}(x,\mathcal{S})^{-2}\] 
for each point $x \in \Sigma$ satisfying $d_{(M,g)}(x,\mathcal{S}) \leq c_1$. 

Suppose now that $y_0$ and $y_1$ are two points in $\Sigma$, and suppose that $\sigma: [0,1] \to \Sigma$ is a smooth path satisfying $\sigma(0) = y_0$ and $\sigma(1) = y_1$. Then 
\[\frac{d}{dt} d_{(M,g)}(\sigma(t),\mathcal{S}) \geq -|\sigma'(t)|_{\hat{g}}\]
whenever $d_{(M,g)}(\sigma(t),\mathcal{S}) \leq c_1$. Here, the derivative is understood in the sense of liminf of backward difference quotients. This implies 
\begin{align*} 
\frac{d}{dt} \big ( d_{(M,g)}(\sigma(t),\mathcal{S})^{-\frac{4}{n-3}} \big ) 
&\leq \frac{4}{n-3} \, d_{(M,g)}(\sigma(t),\mathcal{S})^{-\frac{n+1}{n-3}} \, |\sigma'(t)|_{\hat{g}} \\ 
&\leq \frac{4}{n-3} \, c_2^{-\frac{n+1}{2(n-3)}} \, w(\sigma(t))^{\frac{n+1}{2(n-3)}} \, |\sigma'(t)|_{\hat{g}} \\ 
&= \frac{4}{n-3} \, c_2^{-\frac{n+1}{2(n-3)}} \, |\sigma'(t)|_{\tilde{g}} 
\end{align*} 
whenever $d_{(M,g)}(\sigma(t),\mathcal{S}) \leq c_1$. Here, the derivative is understood in the sense of limsup of backward difference quotients. From this, we deduce that 
\begin{align*} 
d_{(M,g)}(y_1,\mathcal{S})^{-\frac{4}{n-3}} 
&\leq \max \big \{ d_{(M,g)}(y_0,\mathcal{S})^{-\frac{4}{n-3}},c_1^{-\frac{4}{n-3}} \big \} \\ 
&+ \frac{4}{n-3} \, c_2^{-\frac{n+1}{2(n-3)}} \int_0^1 |\sigma'(t)|_{\tilde{g}} \, dt. 
\end{align*}
Thus, 
\begin{align*} 
d_{(M,g)}(y_1,\mathcal{S})^{-\frac{4}{n-3}} 
&\leq \max \big \{ d_{(M,g)}(y_0,\mathcal{S})^{-\frac{4}{n-3}},c_1^{-\frac{4}{n-3}} \big \} \\ 
&+ \frac{4}{n-3} \, c_2^{-\frac{n+1}{2(n-3)}} \, d_{(\Sigma,\tilde{g})}(y_0,y_1), 
\end{align*} 
where $d_{(\Sigma,\tilde{g})}(y_0,y_1)$ denotes the Riemannian distance of $y_0$ and $y_1$ with respect to the metric $\tilde{g}$ on $\Sigma$. Therefore, the metric $\tilde{g}$ on $\Sigma$ is complete. This completes the proof of Proposition \ref{completeness}. \\

\begin{corollary}
\label{singular.set.non.empty}
If $\mathcal{S} \neq \emptyset$, then $(\Sigma,\tilde{g},\tilde{\rho},\tilde{Q})$ is an $(n-1)$-dataset and its mass (in the sense of Definition \ref{mass.of.dataset}) is equal to $0$.
\end{corollary}

\textbf{Proof.} 
It follows from Proposition \ref{completeness} that $(\Sigma,\tilde{g})$ is a complete Riemannian manifold. Using Proposition \ref{quadratic.form.with.respect.to.tilde.g}, we can show that $(\Sigma,\tilde{g},\tilde{\rho},\tilde{Q})$ is an $(n-1)$-dataset. We next observe that the function $\Psi$ vanishes near infinity. This implies $w = 1$ near infinity. Consequently, 
\[\tilde{g} = \hat{g} = dx_1 \otimes dx_1 + \hdots + dx_{n-1} \otimes dx_{n-1} + O \big ( (x_1^2+\hdots+x_{n-1}^2)^{\frac{2-n}{2}} \big )\] 
and 
\[\tilde{\rho} = \hat{\rho} = 1 + O \big ( (x_1^2+\hdots+x_{n-1}^2)^{\frac{2-n}{2}} \big )\] 
near infinity. This shows that $(\Sigma,\tilde{g},\tilde{\rho},\tilde{Q})$ has zero mass (in the sense of Definition \ref{mass.of.dataset}). This completes the proof of Corollary \ref{singular.set.non.empty}. \\

Combining Corollary \ref{singular.set.empty} and Corollary \ref{singular.set.non.empty}, we conclude that Theorem \ref{pmt} is false in dimension $n-1$. This contradicts the inductive hypothesis. The proof of Theorem \ref{pmt} is now complete. 

\appendix

\section{A divergence identity}

In this section, we state a divergence identity that generalizes Proposition A.2 in \cite{Brendle-Hung} (see also \cite{Ambrozio-Carlotto-Sharp}).

\begin{proposition} 
\label{divergence.formula.for.second.variation} 
Let $(M,g)$ be a Riemannian manifold and let $\hat{\rho}$ be a positive function on $M$. Let $\Sigma$ be a two-sided hypersurface in $M$ satisfying $H_\Sigma + \langle \nabla \log \hat{\rho},\nu_\Sigma \rangle = \Phi$, where $\Phi$ is a smooth function on $M$. Let $V$ be a smooth vector field on $M$, and let $W = D_V V$. We define a function $f$ on $\Sigma$ by $f = \langle V,\nu_\Sigma \rangle$. Moreover, we define a tangential vector field $Z$ along $\Sigma$ by 
\[Z = D_{V^{\text{\rm tan}}}^\Sigma (V^{\text{\rm tan}}) - \text{\rm div}_\Sigma(V^{\text{\rm tan}}) \, V^{\text{\rm tan}} + 2 \sum_{k=1}^{n-1} h_\Sigma(V^{\text{\rm tan}},e_k) \, \langle V,\nu_\Sigma \rangle \, e_k.\] 
Then 
\begin{align*} 
&\hat{\rho} \, |\nabla^\Sigma f|^2 - \hat{\rho} \, (\text{\rm Ric}(\nu_\Sigma,\nu_\Sigma) + |h_\Sigma|^2) \, f^2 + (D^2 \hat{\rho})(\nu_\Sigma,\nu_\Sigma) \, f^2 - \hat{\rho}^{-1} \, \langle \nabla \hat{\rho},\nu_\Sigma \rangle^2 \, f^2 \\ 
&+ \text{\rm div}_\Sigma(\hat{\rho} \, W^{\text{\rm tan}}) - \text{\rm div}_\Sigma(\hat{\rho} \, Z) + \text{\rm div}_\Sigma(\langle V^{\text{\rm tan}},\nabla^\Sigma \hat{\rho} \rangle \, V^{\text{\rm tan}}) \\ 
&- \hat{\rho} \, \langle \nabla \Phi,\nu_\Sigma \rangle \, f^2 + \text{\rm div}(\hat{\rho} \, \Phi \, V) \, \langle V,\nu_\Sigma \rangle + \text{\rm div}_\Sigma(\hat{\rho} \, \Phi \, \langle V,\nu_\Sigma \rangle \, V^{\text{\rm tan}}) \\ 
&= \frac{1}{2} \, \hat{\rho} \sum_{k=1}^{n-1} (\mathscr{L}_V \mathscr{L}_V g)(e_k,e_k) + V(V(\hat{\rho})) \\ 
&- \frac{1}{2} \, \hat{\rho} \sum_{k,l=1}^{n-1} (\mathscr{L}_V g)(e_k,e_l) \, (\mathscr{L}_V g)(e_k,e_l) \\ 
&+ \frac{1}{4} \, \hat{\rho} \sum_{k,l=1}^{n-1} (\mathscr{L}_V g)(e_k,e_k) \, (\mathscr{L}_V g)(e_l,e_l) \\ 
&+ V(\hat{\rho}) \sum_{k=1}^{n-1} (\mathscr{L}_V g)(e_k,e_k) 
\end{align*} 
at each point on $\Sigma$. Here, $\{e_1,\hdots,e_{n-1}\}$ denotes a local orthonormal frame on $\Sigma$.
\end{proposition}

\textbf{Proof.} 
We adapt the proof of Proposition A.2 in \cite{Brendle-Hung}. As in \cite{Brendle-Hung}, we write $Z = Z^{(1)} + Z^{(2)}$, where 
\[Z^{(1)} = D_{V^{\text{\rm tan}}}^\Sigma (V^{\text{\rm tan}}) - \text{\rm div}_\Sigma(V^{\text{\rm tan}}) \, V^{\text{\rm tan}}\] 
and 
\[Z^{(2)} = 2 \sum_{k=1}^{n-1} h_\Sigma(V^{\text{\rm tan}},e_k) \, \langle V,\nu_\Sigma \rangle \, e_k.\] 
Using the Gauss equations and the identity $\langle D_{e_k} V,e_l \rangle = \langle D_{e_k} V^{\text{\rm tan}},e_l \rangle + h_\Sigma(e_k,e_l) \, \langle V,\nu_\Sigma \rangle$ for $k,l \in \{1,\hdots,n-1\}$, we compute 
\begin{align*} 
&\text{\rm div}_\Sigma(Z^{(1)}) \\ 
&= \sum_{k,l=1}^{n-1} \langle D_{e_k} V^{\text{\rm tan}},e_l \rangle \, \langle D_{e_l} V^{\text{\rm tan}},e_k \rangle - \sum_{k,l=1}^{n-1} \langle D_{e_k} V^{\text{\rm tan}},e_k \rangle \, \langle D_{e_l} V^{\text{\rm tan}},e_l \rangle \\ 
&+ \text{\rm Ric}_\Sigma(V^{\text{\rm tan}},V^{\text{\rm tan}}) \\ 
&= \sum_{k,l=1}^{n-1} \langle D_{e_k} V,e_l \rangle \, \langle D_{e_l} V,e_k \rangle - \sum_{k,l=1}^{n-1} \langle D_{e_k} V,e_k \rangle \, \langle D_{e_l} V,e_l \rangle \\ 
&- 2 \sum_{k,l=1}^{n-1} h_\Sigma(e_k,e_l) \, \langle D_{e_k} V^{\text{\rm tan}},e_l \rangle \, \langle V,\nu_\Sigma \rangle + 2 \, H_\Sigma \sum_{k=1}^{n-1} \langle D_{e_k} V,e_k \rangle \, \langle V,\nu_\Sigma \rangle \\ 
&- H_\Sigma^2 \, \langle V,\nu_\Sigma \rangle^2 - |h_\Sigma|^2 \, \langle V,\nu_\Sigma \rangle^2 + H_\Sigma \, h_\Sigma(V^{\text{\rm tan}},V^{\text{\rm tan}}) - h_\Sigma^2(V^{\text{\rm tan}},V^{\text{\rm tan}}) \\ 
&+ \sum_{k=1}^{n-1} R(V^{\text{\rm tan}},e_k,V^{\text{\rm tan}},e_k). 
\end{align*}
Using the Codazzi equations, we obtain 
\begin{align*} 
\text{\rm div}_\Sigma(Z^{(2)}) 
&= 2 \sum_{k,l=1}^{n-1} h_\Sigma(e_k,e_l) \, \langle D_{e_k} V^{\text{\rm tan}},e_l \rangle \, \langle V,\nu_\Sigma \rangle \\ 
&+ 2 \sum_{k=1}^{n-1} h_\Sigma(V^{\text{\rm tan}},e_k) \, \langle D_{e_k} V,\nu_\Sigma \rangle + 2 \, h_\Sigma^2(V^{\text{\rm tan}},V^{\text{\rm tan}}) \\ 
&+ 2 \sum_{k=1}^{n-1} R(V^{\text{\rm tan}},e_k,\nu_\Sigma,e_k) \, \langle V,\nu_\Sigma \rangle + 2 \, \langle \nabla^\Sigma H_\Sigma,V^{\text{\rm tan}} \rangle \, \langle V,\nu_\Sigma \rangle. 
\end{align*}
Moreover, 
\[|\nabla^\Sigma f|^2 = \sum_{k=1}^{n-1} \langle D_{e_k} V,\nu_\Sigma \rangle^2 + 2 \sum_{k=1}^{n-1} h_\Sigma(V^{\text{\rm tan}},e_k) \, \langle D_{e_k} V,\nu_\Sigma \rangle + h_\Sigma^2(V^{\text{\rm tan}},V^{\text{\rm tan}}).\] 
Putting these facts together, we obtain 
\begin{align*} 
&\text{\rm div}_\Sigma Z - |\nabla^\Sigma f|^2 + (\text{\rm Ric}(\nu_\Sigma,\nu_\Sigma) + |h_\Sigma|^2) \, f^2 \\ 
&= \sum_{k,l=1}^{n-1} \langle D_{e_k} V,e_l \rangle \, \langle D_{e_l} V,e_k \rangle - \sum_{k,l=1}^{n-1} \langle D_{e_k} V,e_k \rangle \, \langle D_{e_l} V,e_l \rangle \\ 
&- \sum_{k=1}^{n-1} \langle D_{e_k} V,\nu_\Sigma \rangle^2 + \sum_{k=1}^{n-1} R(V,e_k,V,e_k) \\ 
&+ H_\Sigma \, h_\Sigma(V^{\text{\rm tan}},V^{\text{\rm tan}}) - H_\Sigma^2 \, \langle V,\nu_\Sigma 
\rangle^2 \\ 
&+ 2 \, H_\Sigma \sum_{k=1}^{n-1} \langle D_{e_k} V,e_k \rangle \, \langle V,\nu_\Sigma \rangle + 2 \, \langle \nabla^\Sigma H_\Sigma,V^{\text{\rm tan}} \rangle \, \langle V,\nu_\Sigma \rangle. 
\end{align*} 
Using the identity $H_\Sigma + \langle \nabla \log \hat{\rho},\nu_\Sigma \rangle = \Phi$, it follows that 
\begin{align*} 
&\text{\rm div}_\Sigma(\hat{\rho} \, Z) - \text{\rm div}_\Sigma(\langle V^{\text{\rm tan}},\nabla^\Sigma \hat{\rho} \rangle \, V^{\text{\rm tan}}) \\ 
&- \hat{\rho} \, |\nabla^\Sigma f|^2 + \hat{\rho} \, (\text{\rm Ric}(\nu_\Sigma,\nu_\Sigma) + |h_\Sigma|^2) \, f^2 \\ 
&- (D^2 \hat{\rho})(\nu_\Sigma,\nu_\Sigma) \, f^2 + \hat{\rho}^{-1} \, \langle \nabla \hat{\rho},\nu_\Sigma \rangle^2 \, f^2 \\ 
&- 2\hat{\rho} \, \langle \nabla^\Sigma \Phi,V^{\text{\rm tan}} \rangle \, \langle V,\nu_\Sigma \rangle - 2\hat{\rho} \, \Phi \, \text{\rm div}_\Sigma(V^{\text{\rm tan}}) \, \langle V,\nu_\Sigma \rangle \\ 
&- 2 \, \Phi \, \langle \nabla^\Sigma \hat{\rho},V^{\text{\rm tan}} \rangle \, \langle V,\nu_\Sigma \rangle - \hat{\rho} \, \Phi \, h_\Sigma(V^{\text{\rm tan}},V^{\text{\rm tan}}) - \hat{\rho} \, \Phi^2 \, \langle V,\nu_\Sigma \rangle^2 \\ 
&= \hat{\rho} \sum_{k,l=1}^{n-1} \langle D_{e_k} V,e_l \rangle \, \langle D_{e_l} V,e_k \rangle - \hat{\rho} \sum_{k,l=1}^{n-1} \langle D_{e_k} V,e_k \rangle \, \langle D_{e_l} V,e_l \rangle \\ 
&- \hat{\rho} \sum_{k=1}^{n-1} \langle D_{e_k} V,\nu_\Sigma \rangle^2 + \hat{\rho} \sum_{k=1}^{n-1} R(V,e_k,V,e_k) \\ 
&- 2 \, V(\hat{\rho}) \sum_{k=1}^{n-1} \langle D_{e_k} V,e_k \rangle - (D^2 \hat{\rho})(V,V). 
\end{align*} 
We next observe that 
\[(\mathscr{L}_V \mathscr{L}_V g)(X,Y) - (\mathscr{L}_W g)(X,Y) = 2 \, \langle D_X V,D_Y V \rangle - 2 \, R(V,X,V,Y)\] 
for all vector fields $X,Y$ on $M$. Moreover, 
\[V(V(\hat{\rho})) - W(\hat{\rho}) = (D^2 \hat{\rho})(V,V).\] 
Using these identities together with the identity $H_\Sigma + \langle \nabla \log \hat{\rho},\nu_\Sigma \rangle = \Phi$, we obtain 
\begin{align*} 
&\frac{1}{2} \, \hat{\rho} \sum_{k=1}^{n-1} (\mathscr{L}_V \mathscr{L}_V g)(e_k,e_k) + V(V(\hat{\rho})) - \text{\rm div}_\Sigma(\hat{\rho} \, W^{\text{\rm tan}}) - \hat{\rho} \, \Phi \, \langle W,\nu_\Sigma \rangle \\ 
&= \hat{\rho} \sum_{k=1}^{n-1} |D_{e_k} V|^2 - \hat{\rho} \sum_{k=1}^{n-1} R(V,e_k,V,e_k) + (D^2 \hat{\rho})(V,V) \\  
&= \hat{\rho} \sum_{k,l=1}^{n-1} \langle D_{e_k} V,e_l \rangle^2 + \hat{\rho} \sum_{k=1}^{n-1} \langle D_{e_k} V,\nu_\Sigma \rangle^2 - \hat{\rho} \sum_{k=1}^{n-1} R(V,e_k,V,e_k) + (D^2 \hat{\rho})(V,V). 
\end{align*} 
Putting these facts together, we conclude that 
\begin{align*} 
&\frac{1}{2} \, \hat{\rho} \sum_{k=1}^{n-1} (\mathscr{L}_V \mathscr{L}_V g)(e_k,e_k) + V(V(\hat{\rho})) - \text{\rm div}_\Sigma(\hat{\rho} \, W^{\text{\rm tan}}) \\ 
&+ \text{\rm div}_\Sigma(\hat{\rho} \, Z) - \text{\rm div}_\Sigma(\langle V^{\text{\rm tan}},\nabla^\Sigma \hat{\rho} \rangle \, V^{\text{\rm tan}}) \\ 
&- \hat{\rho} \, |\nabla^\Sigma f|^2 + \hat{\rho} \, (\text{\rm Ric}(\nu_\Sigma,\nu_\Sigma) + |h_\Sigma|^2) \, f^2 \\ 
&- (D^2 \hat{\rho})(\nu_\Sigma,\nu_\Sigma) \, f^2 + \hat{\rho}^{-1} \, \langle \nabla \hat{\rho},\nu_\Sigma \rangle^2 \, f^2 \\ 
&- \hat{\rho} \, \Phi \, \langle W,\nu_\Sigma \rangle - 2\hat{\rho} \, \langle \nabla^\Sigma \Phi,V^{\text{\rm tan}} \rangle \, \langle V,\nu_\Sigma \rangle - 2\hat{\rho} \, \Phi \, \text{\rm div}_\Sigma(V^{\text{\rm tan}}) \, \langle V,\nu_\Sigma \rangle \\ 
&- 2 \, \Phi \, \langle \nabla^\Sigma \hat{\rho},V^{\text{\rm tan}} \rangle \, \langle V,\nu_\Sigma \rangle - \hat{\rho} \, \Phi \, h_\Sigma(V^{\text{\rm tan}},V^{\text{\rm tan}}) - \hat{\rho} \, \Phi^2 \, \langle V,\nu_\Sigma \rangle^2 \\ 
&= \hat{\rho} \sum_{k,l=1}^{n-1} \langle D_{e_k} V,e_l \rangle^2 + \hat{\rho} \sum_{k,l=1}^{n-1} \langle D_{e_k} V,e_l \rangle \, \langle D_{e_l} V,e_k \rangle \\ 
&- \hat{\rho} \sum_{k,l=1}^{n-1} \langle D_{e_k} V,e_k \rangle \, \langle D_{e_l} V,e_l \rangle - 2 \, V(\hat{\rho}) \sum_{k=1}^{n-1} \langle D_{e_k} V,e_k \rangle. 
\end{align*}
Finally, a straightforward calculation gives  
\begin{align*} 
&\hat{\rho} \, \Phi \, \langle W,\nu_\Sigma \rangle + 2\hat{\rho} \, \langle \nabla^\Sigma \Phi,V^{\text{\rm tan}} \rangle \, \langle V,\nu_\Sigma \rangle + 2\hat{\rho} \, \Phi \, \text{\rm div}_\Sigma(V^{\text{\rm tan}}) \, \langle V,\nu_\Sigma \rangle \\ 
&+ 2 \, \Phi \, \langle \nabla^\Sigma \hat{\rho},V^{\text{\rm tan}} \rangle \, \langle V,\nu_\Sigma \rangle + \hat{\rho} \, \Phi \, h_\Sigma(V^{\text{\rm tan}},V^{\text{\rm tan}}) + \hat{\rho} \, \Phi^2 \, \langle V,\nu_\Sigma \rangle^2 \\ 
&= -\hat{\rho} \, \langle \nabla \Phi,\nu_\Sigma \rangle \, \langle V,\nu_\Sigma \rangle^2 + \text{\rm div}(\hat{\rho} \, \Phi \, V) \, \langle V,\nu_\Sigma \rangle + \text{\rm div}_\Sigma(\hat{\rho} \, \Phi \, \langle V,\nu_\Sigma \rangle \, V^{\text{\rm tan}}). 
\end{align*}
From this, the assertion follows easily. This completes the proof of Proposition \ref{divergence.formula.for.second.variation}. \\

\end{document}